\documentclass[a4paper,11pt,reqno,twoside]{amsart}

\usepackage{array,graphics,color}
\newenvironment{gap}
  {\color{red}}%
  {}%

\usepackage{amssymb}
\usepackage[UKenglish]{babel}
\usepackage[bookmarks=true,
           bookmarksnumbered=true,
           colorlinks=false,
           pdfstartview=FitH]{hyperref}
\hypersetup{pdftitle={Quantum Feynman-Kac perturbations},
            pdfauthor={Alexander C. R. Belton, J. Martin Lindsay and Adam G. Skalski}}

\theoremstyle{plain}
\newtheorem{propn}{Proposition}[section]
\newtheorem{thm}[propn]{Theorem}
\newtheorem{lemma}[propn]{Lemma}
\newtheorem{cor}[propn]{Corollary}

\theoremstyle{definition}
\newtheorem{defn}[propn]{Definition}
\newtheorem{example}[propn]{Example}

\theoremstyle{remark}
\newtheorem*{rem}{Remark}
\newtheorem*{rems}{Remarks}

\newcommand{\tinymaths}[1]{\mbox{\small $#1$}}

\newcommand{\iotaAFock}{\iota_{\Fock}^{\Al}}
\newcommand{\iotaAnoise}{\iota_{\noise}^{\Al}}
\newcommand{\iotaAkhat}{\iota_{\khat}^{\Al}}
\newcommand{\Noise}{\mathsf{N}}
\newcommand{\grad}{\nabla}
\newcommand{\gradhat}{\wh{\grad}}
\newcommand{\Y}[1]{Y^{#1}}

\newcommand{\jF}{j^{F}}
\newcommand{\jFonetwo}{j^{F_1, F_2}}
\newcommand{\CkA}{\mathfrak{qc}}
\newcommand{\CbetakA}{\CkA_\beta}
\newcommand{\CzerokA}{\CkA_0}
\newcommand{\fhat}{\wh{f}}
\newcommand{\ghat}{\wh{g}}

\newcommand{\cD}{\mathcal{D}}
\newcommand{\calDzerozero}{\cD^{0, 0}}
\newcommand{\calDcd}{\cD^{c, d}}
\newcommand{\calDzerozerotwo}{\calDzerozero_2}
\newcommand{\chat}{\wh{c}}
\newcommand{\dhat}{\wh{d}}

\newcommand{\ve}{\varepsilon}
\newcommand{\al}{\alpha}
\newcommand{\vp}{\varpi}

\newcommand{\hil}{\mathsf{h}}
\newcommand{\Hil}{\mathsf{H}}
\newcommand{\Kil}{\mathsf{K}}
\newcommand{\Al}{\mathsf{A}}
\newcommand{\init}{\mathfrak{h}}
\newcommand{\noise}{\mathsf{k}}
\newcommand{\khat}{{\wh{\noise}}}
\newcommand{\Fock}{\mathcal{F}}
\newcommand{\Exps}{\mathcal{E}}
\newcommand{\Step}{\mathbb{S}}

\newcommand{\Real}{\mathbb{R}}
\newcommand{\Rplus}{\Real_+}
\newcommand{\Comp}{\mathbb{C}}

\newcommand{\ip}[2]{\langle #1, #2 \rangle}
\newcommand{\norm}[1]{\lVert #1 \rVert}
\newcommand{\cbnorm}[1]{\norm{#1}_\cb}
\newcommand{\bra}[1]{\langle #1 |}
\newcommand{\ket}[1]{| #1 \rangle}

\newcommand{\cb}{{\text{\tu{cb}}}}

\newcommand{\wh}{\widehat}
\newcommand{\wt}{\widetilde}
\newcommand{\ot}{\otimes}
\newcommand{\otol}{\mathbin{\overline{\ot}}}
\newcommand{\otul}{\mathbin{\underline{\ot}}}

\newcommand{\op}{\oplus}
\newcommand{\les}{\leqslant}
\newcommand{\ges}{\geqslant}
\newcommand{\uwlim}{\mathop{\mathrm{uw\,lim}}}
\newcommand{\comp}{\mathbin{\circ}}

\newcommand{\tu}{\textup}

 \DeclareMathOperator{\uw}{uw}   
\DeclareMathOperator{\Dom}{Dom}
\DeclareMathOperator{\Ran}{Ran}
\DeclareMathOperator{\id}{id}
\DeclareMathOperator{\re}{Re}

\newenvironment{alist}
{

\begin{enumerate}}
{\end{enumerate}}

\newenvironment{rlist}
{

\begin{enumerate}}
{\end{enumerate}}

\numberwithin{equation}{section}
\begin{document}

\title[Quantum Feynman--Kac perturbations]%
{Quantum Feynman--Kac perturbations}
\author[Alexander Belton, Martin Lindsay and Adam Skalski]%
{Alexander C. R. Belton, J. Martin Lindsay and Adam G. Skalski}
\address{Department of Mathematics and Statistics\\Fylde College\\
Lancaster University\\Lancaster LA1 4YF\\
United Kingdom}
\email{a.belton@lancaster.ac.uk}
\email{j.m.lindsay@lancaster.ac.uk}
\address{Institute of Mathematics\\Polish Academy of Sciences\\
ul.~\'{S}niadeckich~8\\00-956 Warszawa\\
Poland}
\email{a.skalski@impan.pl}

\dedicatory{Dedicated to the memory of Bill Arveson}




\begin{abstract}
We develop fully noncommutative Feynman--Kac formulae by employing
quantum stochastic processes. To this end we establish some theory
for perturbing quantum stochastic flows on von~Neumann algebras by
multiplier cocycles.
 Multiplier cocycles are constructed via
 quantum stochastic differential equations
 whose coefficients are driven by the flow.
 The resulting class of cocycles is characterised under alternative
 assumptions of separability or Markov regularity.
 Our results generalise those obtained using classical Brownian
 motion on the one hand, and results for unitarily implemented
 flows on the other.
\end{abstract}

\maketitle

\section*{Introduction}\label{intro}
Feynman--Kac formulae for vector field-type perturbations of a class
of noncommutative elliptic operators were developed in~\cite{LiS}
and extended in~\cite{BP}. In those papers classical Brownian motion
is employed, so the noncommutativity is confined to the operator
algebra which replaces a function space in the classical
Feynman--Kac formula. Moreover the unperturbed semigroup in those
papers is a Gaussian average of a unitarily implemented automorphism
group. In this paper we obtain fully noncommutative Feynman--Kac
formulae by employing \emph{quantum} stochastic processes
(\cite{Partha},\cite{Meyer},\cite{jmlLectures},\cite{SiG}). To this
end we develop the theory of perturbing quantum stochastic flows by
multiplier cocycles, in particular those governed by quantum
stochastic differential equations
 (\textit{cf}.~\cite{EvH},\cite{DaS},\cite{GLW}).
We also show that every sufficiently regular multiplier cocycle is
governed by a quantum stochastic differential equation, whose
coefficients are driven by the free flow;
 in particular, this extends results of~\cite{Bra}
and~\cite{LWjfa}.

Another route to the Feynman--Kac formulae obtained here is sketched
in~\cite{BLS}. This starts from the observation that Bahn and Park's
ideas in~\cite{BP} may be expressed naturally in vacuum-adapted
quantum stochastic calculus ([$\text{B}_{1,2}$]).

Vector field-type perturbations of the Laplacian on $\Real^n$ were
realised with classical Brownian motion by Parthasarathy and Sinha
(\cite{PS0}). Early work on noncommutative Feynman--Kac formulae was
done by Accardi, Frigerio and Lewis
 (\cite{accardi},\cite{AcF},\cite{AFL}) and by
Hudson, Ion and Parthasarathy (\cite{HIP}); see also~\cite{Hud}
and~\cite{F1}. The operator-algebraic structure of the classical
Feynman--Kac formula was elucidated by Arveson (\cite{ArvFK}).

The plan of the paper is as follows. In Section~\ref{section:qsa} we
give the necessary background material from quantum stochastic
analysis, including some recent results, such as Theorem~\ref{XYZ}
and the converse part of Theorem~\ref{amalgam}, to be applied in the
later sections. This section also serves to introduce notation and
terminology. In Section~\ref{section:multipliers} we introduce
adapted multiplier cocycles for a fixed quantum stochastic flow,
here referred to as the \emph{free flow}. These give rise to
completely bounded quantum stochastic cocycles, which we view as
perturbations of the free flow. The relevant existence and
uniqueness results for quantum stochastic differential equations
with time-dependent coefficients are recalled in
Section~\ref{section:qsdes}, where H\"older properties of their
solutions are also discussed. With these, perturbation processes are
constructed in Section~\ref{section:perturb} and their
contractivity, isometry and coisometry properties are characterised
in terms of the driving coefficients of the quantum stochastic
differential equation. In Section~\ref{section:bounded pert} bounded
perturbation processes having locally uniform bounds are shown to be
multiplier cocycles for the free flow.
 In Section~\ref{section:character} the class of
 quasicontractive perturbation processes is characterised,
 assuming that the system and noise dimension spaces are
 separable;
 a further characterisation is obtained
 under the alternative assumption that the free flow is Markov
regular. In Section~\ref{section:pert flows} the perturbation theorem
for quantum stochastic flows is deduced from results of
 Section~\ref{section:bounded pert}, by application of the quantum It\^o
product formula. In the final section we obtain quantum Feynman--Kac
formulae for quantum Markov semigroups via their quantum stochastic
dilation.

\medskip\noindent
\emph{General notation.} All Hilbert spaces here are complex, with
inner products linear in their second argument. We usually
abbreviate the simple tensor $u \ot \xi$ of Hilbert-space vectors $u
\in \hil$ and $\xi \in \Hil$ to $u\xi$. We write $\ket{\hil}$ and
$\bra{\hil}$ respectively for the column and row operataor spaces
$B( \Comp; \hil )$ and $B( \hil; \Comp )$, with $\ket{u}$ and
$\bra{u}$ denoting respectively the operator $\Comp \to \hil, \
\lambda \mapsto \lambda u$ and the linear functional $\hil \to
\Comp, \ v \mapsto \ip{u}{v}$. We also use the notation
\begin{equation}\label{E-notation}
E_u := I_{\hil'} \ot \ket{u} \ot I_\hil \quad \text{and} \quad %
E^u := I_{\hil'} \ot \bra{u} \ot I_\hil,
\end{equation}
where context dictates the choice of Hilbert spaces $\hil'$ and
$\hil$. Algebraic tensor products are denoted $\otul$ and ultraweak
tensor products $\otol$. For a Hilbert space $\hil$, $\iota_\hil$
denotes the ampliation $B(\Hil) \to B(\Hil \ot \hil)$,
 $T \mapsto T \ot I_\hil$ or $B(\Hil) \to B(\hil \ot \Hil)$,
 $T \mapsto I_\hil \ot T$, where the order and Hilbert space $\Hil$
 depend on context, and, for a von Neumann algebra $\Al$,
 $\iota_\hil^\Al$ denotes the induced ampliation
 $\Al \to \Al \otol B( \hil )$, or $\Al \to B( \hil ) \otol \Al$.

For $u$,~$v \in \hil$, $\omega_{u, v}$ denotes both the vector
functional $x \mapsto \ip{u}{x v}$ on $B( \hil )$ and its
restrictions to von~Neumann subalgebras. Finally, for a
vector-valued function $f: \Rplus \to V$ and a subinterval $J$ of
$\Rplus$, the function which agrees with~$f$ on $J$ and is zero on
$\Rplus \setminus J$ is denoted $f_J$, and, for $c \in V$, the
function which equals $c$ on $J$ and the zero vector on $\Rplus
\setminus J$ is denoted $c_J$.

\section{Quantum stochastic analysis}\label{section:qsa}

Fix now, and for the rest of the paper, a von~Neumann algebra $\Al$
acting faithfully on a Hilbert space $\init$ and a further Hilbert
space $\noise$, and set $\khat = \Comp \op \noise$. Under the natural
identification $\khat \ot \init = \init \op ( \noise \ot \init )$,
elements of $B( \khat ) \otol \Al$ take the block-matrix form
$\left[\begin{smallmatrix}
 k & m\\[0.5ex]
 l & n
\end{smallmatrix}\right]$, where $k \in \Al$,
$l \in \ket{\noise} \otol \Al$, $m \in \bra{\noise} \otol \Al$ and
$n\in B( \noise ) \otol \Al$.

In this section we provide the necessary background material in
quantum stochastic (QS) analysis. Further detail may be found
in~\cite{jmlLectures}, but some more recent results are also included
below.

\begin{defn}
A \emph{Markov semigroup} on $\Al$ is a pointwise ultraweakly
continuous semigroup of normal positive unital maps on $\Al$.
\end{defn}

\begin{rem}
If unitality is relaxed to contractivity then the semigroup is called
\emph{sub-Markov}. All the sub-Markov semigroups appearing here will
be completely positive. Completely positive sub-Markov semigroups are
also known as \emph{quantum dynamical semigroups}.
\end{rem}

For a subinterval $J$ of $\Rplus$, let $\Fock_J$ denote the
symmetric Fock space over $L^2( J ; \noise )$ and set $\Noise_J :=
B( \Fock_J)$, shortened to $\Fock$ and $\Noise$ when $J = \Rplus$;
write $I_J$ for the identity operator on $\Fock_J$. Thus, for $0
\les r < t \les \infty$,
\begin{equation}\label{decomp}
\Fock = %
\Fock_{[ 0, r [} \ot \Fock_{[ r, t [} \ot \Fock_{[ t, \infty [}
\quad \text{and} \quad %
\Noise = \Noise_{[ 0, r [} \otol \Noise_{[ r, t [} \otol %
\Noise_{[ t, \infty [}.
\end{equation}
The subspace of $L^2( \Rplus; \noise )$ consisting of step functions
is denoted $\Step$; for elements of $\Step$ we always take their
right-continuous version. We use normalised exponential vectors,
\begin{equation}\label{varpi}
\vp( f ) := e^{-\frac{1}{2} \norm{f}^2} \ve( f ) = %
e^{-\frac{1}{2} \norm{f}^2} %
\bigl( (n!)^{-1/2} f^{\ot n} \bigr)_{n \geq 0} \quad ( f \in \Step ),
\end{equation}
and let $\Exps$ denote their linear span. We refer to the map
\begin{equation}\label{vacuum}
\mathbb{E} := \id_\Al \otol \omega_{\vp( 0 ),\vp( 0 )} : %
\Al \otol \Noise \to \Al, \ T \mapsto E^{\vp( 0 )} T E_{\vp( 0 )}
\end{equation}
as the \emph{vacuum-expectation map}.

In terms of the unitary operator
 $T_r : \Fock \to \Fock_{[ r, \infty [}$ fixed by the requirement
\[
\vp(f) \mapsto \vp( f_r ) \quad ( f \in \Step ),
\]
where $r \in \Rplus$ and $f_r( s ) := f( s - r )$ for all $s \in [
r,\infty [$, the \emph{ampliated CCR flow on} $\Al \otol \Noise$ is
the semigroup of normal unital *-endomorphisms $( \wt{\sigma}_r )_{r
\ges 0}$ determined by the identity
\begin{equation}\label{sigma}
\wt{\sigma}_r( a \ot z ) = a \ot I_{[ 0, r [} \ot T_r z T_r^* %
\quad ( a \in \Al, \, z \in \Noise ).
\end{equation}
We also use the normal *-isomorphisms
 $\sigma_r :
 \Al \otol \Noise \to \Al \otol \Noise_{[ r, \infty [}$
 $(r \ges 0)$ determined by the identity
\[
\sigma_r ( a \ot z ) =  a \ot T_r z T_r^* %
\quad ( a \in \Al, \, z \in \Noise ).
\]

\begin{defn}
 By an \emph{operator process}, or
 \emph{process on} $\init$, we mean a family of
 (possibly unbounded) operators
$X = ( X_t )_{t \ges 0}$ on $\init\ot\Fock$, with common domain
$\init \otul \Exps$, which is adapted, in the sense that
\begin{equation*}
E^{\vp( f )} X_t E_{\vp( g )} = %
\langle \vp( f_{[ t, \infty [} ), \vp( g_{[ t, \infty [} ) \rangle
E^{\vp( f_{[ 0, t [} )} X_t E_{\vp( g_{[ 0, t [} )}
\end{equation*}
($f, g \in \Step, \, t \in \Rplus$), and weakly measurable, in that
\[
t \mapsto \langle \zeta', X_t \zeta \rangle \text{ is measurable } %
\quad ( \zeta' \in \init \ot \Fock, \, \zeta \in \init \otul \Exps).
\]
The process $X$ is \emph{bounded} if each $X_t$ is a bounded
operator; it is a \emph{process in} $\Al$ if the corresponding
matrix elements lie in $\Al$: that is, $E^{\vp( f )} X_t E_{\vp( g
)} \in \Al$ for all $f, g \in \Step$ and $t \in \Rplus$. Finally, we
call the process $X$ \emph{measurable} or \emph{continuous} if, for
all $\zeta \in \init \otul \Exps$, the map $t \mapsto X_t \zeta$ is
strongly measurable or continuous, respectively.
\end{defn}

\begin{rems}
 In the literature it is more usual to allow the common domain for
 an operator process to be of the form $\mathcal{D} \otul \Exps$ where
 $\mathcal{D}$ is a dense subspace of $\init$.
 Here $\init \otul \Exps$ suffices.

When $\init$ and $\noise$ are both separable, measurability is
automatic, by Pettis' Theorem.
\end{rems}

The following notation and terminology are convenient for expressing
the basic estimates and formulae of QS calculus. For an operator
process $X$ we set
\begin{equation}\label{X tilde}
\wt{X}_t := I_{\khat} \otul X_t \quad ( t \in \Rplus ).
\end{equation}
For all $s \in \Rplus$, the operator $\gradhat_s$
 on $\init \ot \Fock$ with domain $\init \otul \Exps$
 is given by linear extension of the prescription
\[
\gradhat_s u \vp( f ) = \fhat( s ) u \vp( f ) %
\quad (u \in \init, \, f \in \Step).
\]
 This is well defined since $f$ is right continuous. The
\emph{quantum It\^o projection} is the operator
\[
\Delta := \begin{bmatrix} 0 & 0\\[1ex] 0 & I_\noise \end{bmatrix}
\in B( \khat ) = \Comp \op B(\noise),
\]
which will appear in ampliated form below, with the same notation.
For \mbox{$f \in \Step$} and any subinterval $J$ of $\Rplus$, the
relevant constants for the estimates~\eqref{FE} and~\eqref{FHE}
below are
\begin{equation}\label{C I f}
C( f ) := 1 + \norm{f} \quad \text{and} \quad %
C( J, f ):= \sqrt{ | J | + C( f_J )^2},
\end{equation}
where $| J |$ is the length of $J$; see~\cite{Lqsi}, Theorem~3.4.

A family of (possibly unbounded) operators $(G_t)_{t\ges 0}$ on
$\khat \ot \init \ot \Fock$ is a \emph{QS-integrable process} if it
satisfies the domain and integrability conditions
\begin{rlist}
\item $\gradhat_t \zeta \in \Dom G_t$ ,
\item $s \mapsto \Delta^\perp G_s \gradhat_s \zeta $ is locally
integrable and
\item $s \mapsto \Delta G_s \gradhat_s \zeta $ is locally
square-integrable,
\end{rlist}
for all $\zeta \in \init \otul \Exps$ and $t \in \Rplus$, and also the
adaptedness condition
 \begin{itemize}
 \item[(iv)]
 for all $f, g \in \Step$ and $t \in \Rplus$,
 \begin{equation}
 E^{\vp(f)} G_t \gradhat_t E_{\vp(g)} = %
 \ip{\vp(f_{[ t, \infty [})}{\vp(g_{[ t, \infty [})} %
  E^{\fhat( t ) \vp(f_{[ 0, t [})}
  G_t E_{\ghat( t ) \vp(g_{[ 0, t [})}.
 \end{equation}
 \end{itemize}

\begin{rems}
These conditions imply that~\eqref{FE} below is finite.

Like the collection of operator processes, the collection of
QS-integrable processes forms a linear space.

The following conditions on an adapted family of operators
$(G_t)_{t\ges 0}$ on $\khat \ot \init \ot \Fock$ are sufficient for
QS integrability:
\begin{alist}
\item $\Dom G_t \supset \khat \otul \init \otul \Exps$ for all
$t \in \Rplus$,
\item $s \mapsto G_s \xi$ is continuous for all
$\xi \in \khat \otul \init \otul \Exps$.
\end{alist}
\end{rems}
In particular, for a locally integrable process $X$ on $\init$, the
adapted family of operators~\eqref{X tilde} is a QS-integrable
process.

Let $X = \bigl(X_0 + \int_0^t G_s \, d \Lambda_s \bigr)_{t \ges 0}$
for a QS-integrable process $G$ and operator $X_0 \in B(\init) \otol
\Noise$. Then $X$ is a continuous process on $\init$ such that, for
all $0 \les r \les t$ and $\zeta \in \init \otul \Exps$, the
\emph{First Fundamental Formula} holds:
\begin{equation}\label{FFF}
\ip{\zeta}{( X_t - X_r ) \zeta} = %
\int_r^t \ip{\gradhat_s \zeta}{G_s \gradhat_s \zeta} \, d s.
\end{equation}
Furthermore, the following hold: the \emph{Second Fundamental Formula},
\begin{equation}\label{SFF}
\norm{X_t \zeta}^2 - \norm{X_r \zeta}^2 = \int_r^t \bigl( 2 \re \ip{\wt{X}_s %
\gradhat_s \zeta}{G_s \gradhat_s \zeta} + %
\norm{\Delta G_s \gradhat_s \zeta}^2 \bigr) \, d s,
\end{equation}
the \emph{Fundamental Estimate},
\begin{align}
& \norm{( X_t - X_r ) u \vp(f)} \notag \\
& \les \int_r^t %
\norm{\Delta^\perp G_s \fhat( s ) u \vp( f )} \, d s
+ C \bigl( f_{[ r, t [} \bigr) \Bigl\{ \int_r^t %
\norm{\Delta G_s \fhat( s ) u \vp( f )}^2 \, d s \Bigr\}^{1 / 2}
\label{FE} \\
& \les C( [ r, t [, f ) \Bigl\{ \int_r^t %
\norm{ G_s \fhat( s ) u \vp( f )}^2 \, d s \Bigr\}^{1 / 2} %
\quad ( u \in \init, \ f \in \Step) \notag
\end{align}
and the \emph{Fundamental H\"older Estimate},
\begin{equation}\label{FHE}
( t - r )^{-1 / 2} \norm{( X_t - X_r ) E_{\vp(f)}} \les %
C( [ r, t [, f ) \sup_{s \in [ r, t [} %
\norm{G_s E_{\fhat( s )} E_{\vp(f)}} \quad ( f \in \Step ).
\end{equation}

Let
$H = \bigl( \wt{X}_t^* G_t + G_t^* \wt{X}_t + %
G_t^* \Delta G_t \bigr)_{t \ges 0}$.
If $\Dom H_t \gradhat_t \supset \init \otul \Exps$ for each
$t \in \Rplus$ (for example, if $X$ and $G$ are both bounded) then the
identity~\eqref{SFF} may be re-expressed as follows:
\begin{equation}\label{SFF'}
\norm{X_t \zeta}^2 - \norm{X_r \zeta}^2 = %
\int_r^t \ip{\gradhat_s\zeta}{H_s \gradhat_s \zeta} \, d s.
\end{equation}
Thus if $H$ is a QS-integrable process then~\eqref{SFF'}
and~\eqref{FFF} combine to yield the \emph{Quantum It\^o Product Formula},
\begin{equation}\label{QIPF}
 X^*_t X_t = X^*_r X_r + \int_r^t H_s \, d\Lambda_s.
\end{equation}

The formulae~\eqref{FFF}, \eqref{SFF}, \eqref{SFF'} and \eqref{QIPF}
may all be polarised; moreover, if $G^*$ is QS integrable then
\[
  X^*_t \supset X^*_r + \int_r^t G_s^* \, d\Lambda_s.
\]

We shall need the following extension of the First Fundamental
Formula; the boundedness assumptions hold in cases of interest.

\begin{lemma}\label{FFF extension}
Let $G$ be a QS-integrable process, let $f$,~$g \in \Step$ and $0
\les r < t$, and suppose that
\begin{alist}
\item
$E^{\vp( f )} \bigl( \int_r^t G_s \, d \Lambda_s \bigr) E_{\vp( g
)}$ is bounded\tu{;}
\item
$E^{\fhat( s )\vp( f )} G_s E_{\ghat( s )\vp( g )}$ is bounded,
 uniformly for $s \in [ r, t [$.
\end{alist}
Then,
 for all $\omega \in B( \init )_*$,
\[
\omega\Bigl(
 E^{\vp( f )} \int_r^t G_s \, d \Lambda_s \ E_{\vp( g )}
\Bigr) = \int_r^t
\omega\bigl(
 E^{\fhat( s )\vp( f )} G_s E_{\ghat( s )\vp( g )}
\bigr) \, d s.
\]
\end{lemma}

\begin{proof}
When $\omega$ is a vector functional $\omega_{u, v}$ the identity is
equivalent to the First Fundamental Formula~\eqref{FFF}. For a general
normal linear functional the result follows from Lebesgue's Dominated
Convergence Theorem and the norm totality of vector functionals in
$B( \init )_*$.
\end{proof}

\begin{defn}
A \emph{mapping process}, or
 \emph{process on $\Al$}, is
a family $k = ( k_t )_{t \ges 0}$ of linear maps with common domain
$\Al$, such that $\bigl( k_t( a ) \bigr)_{t \ges 0}$ is a process in
$\Al$ for all $a \in \Al$.

A process $k$ on $\Al$ is \emph{bounded}, \emph{completely bounded},
\emph{completely positive} or \emph{normal} if $k_t$ has that
property for each $t \in \Rplus$, and is \emph{continuous} if the
operator process $\bigl( k_t( a ) \bigr)_{t\ges 0}$ is continuous
for each $a \in \Al$. We call a bounded process~$k$ on $\Al$
\emph{ultraweakly continuous} if it is pointwise ultraweakly
continuous: $t \mapsto k_t( a )$ is ultraweakly continuous for each
$a \in \Al$.
\end{defn}

For QS differential equations and QS cocycles, to be defined next,
various forms of ampliation are needed.

Let $k$ be a normal completely bounded process on $\Al$. Then
adaptedness implies that, for all $r \ges 0$ and all $a \in \Al$,
\[
k_r( a ) = k_{r)}( a ) \ot I_{[ r, \infty [},
\]
where $k_{r)}( a )$ is an operator in $\Al \otol \Noise_{ [ 0, r
[}$. Each map
\[
k_{r)} : \Al \to \Al \otol \Noise_{[ 0, r [}
\]
is normal and completely bounded, so we may define
\[
\wh{k}_r := k_{r)} \otol \id_{\Noise_{[ r, \infty [}} : %
\Al \otol \Noise_{[ r, \infty [} \to %
\Al \otol \Noise_{[ 0, r [} \otol \Noise_{[ r, \infty [} = %
\Al \otol \Noise.
\]
Also define a normal completely bounded process $\wt{k}$ on
$B( \khat ) \otol \Al$ by setting
\[
\wt{k}_t := \id_{B( \khat )} \otol \, k_t \quad ( t \in \Rplus ).
\]

\begin{defn}
A normal completely bounded process $k$ on $\Al$ is a \emph{QS
cocycle} if
\begin{equation}\label{cocycle relations}
k_0 = \iotaAFock \quad \text{and} \quad %
k_{r + t} = \wh{k}_r \comp \sigma_r \comp k_t \quad %
( r, t \in \Rplus );
\end{equation}
it is a \emph{QS flow} if furthermore $k$ is ultraweakly continuous,
*-homomorphic and unital.
\end{defn}

 More generally, a process $k$ on $\Al$ for which the maps
\begin{equation}
\label{kfg}
\kappa_t^{f,g} := %
E^{\vp(f_{[ 0, t [})} k_t( \cdot ) E_{\vp(g_{[ 0, t [})}
\quad ( f, g \in \Step, \, t \in \Rplus )
\end{equation}
are bounded is a QS cocycle if it satisfies the \emph{weak cocycle
relations}
\begin{equation}\label{weak cocycle}
\kappa_0^{f,g} = \id_\Al \quad \mbox{and} \quad %
\kappa_{r + t}^{f, g} = %
\kappa_r^{f, g} \comp \kappa_t^{S^*_r f, S^*_r g} \quad %
( f, g \in \Step, \, r, t \in \Rplus ),
\end{equation}
where $S_r$ is the right-shift isometry on $L^2( \Rplus; \noise )$
so that
\[
( S_r f )( u ) = 1_{[ r, \infty [}( u ) f( u - r ) \quad %
( u \in \Rplus ).
\]
The weak cocycle relation is equivalent to~\eqref{cocycle relations}
when~$k$ is completely bounded and normal.

Such processes arise naturally; see Theorem~\ref{amalgam}.

\begin{lemma}[{\textit{Cf.}~\cite{LWjfa}, Proposition~4.3}]
\label{k to K} Let $k$ be a normal completely bounded process on
$\Al$ and consider the family $K := \bigl( \wh{k}_t \comp \sigma_t
\bigr)_{t\ges 0}$ of normal completely bounded maps on $\Al \otol
\Noise$.
\begin{alist}
\item For each $t\in\Rplus$,
$\norm{K_t}_\cb = \norm{k_t}_\cb$. Moreover, $K$ is unital or
*-homomorphic if and only if $k$ has the same property.
\item The following are equivalent\tu{:}
\begin{rlist}
\item $k$ is a QS cocycle\tu{;}
\item $K$ is a semigroup.
\end{rlist}
\end{alist}
\end{lemma}
\begin{proof}
(a) Let $t\in \Rplus$. Since $\iotaAFock$ and $\sigma_t$ are both
unital and *-homomorphic, and thus completely isometric, the claim
follows from the definition of $K_t$ and the identity
\begin{equation}\label{k = K iota}
k_t = K_t \circ \iotaAFock.
\end{equation}
(b) Let $s,t\in \Rplus$. For a simple tensor $A = a \ot z$, the
identity
\[
K_s\bigl( A X \bigr) = K_s( A ) \wt{\sigma}_s( X ) \quad %
( A \in \Al \otol \Noise, \, X \in 1_{\Al} \ot \Noise )
\]
holds, since both sides equal $k_s( a ) \wt{\sigma}_s( Z X )$, where
$Z = 1_{\Al} \ot z$; it therefore holds in general, by linearity and
normality. In particular, if (i) holds then, for $a\in\Al$ and $z
\in \Noise$,
\begin{align*}
( K_s \circ K_t )( a \ot z ) & = %
K_s\bigl( k_t( a ) \wt{\sigma}_t( 1_\Al \ot z ) \bigr) \\
& = K_s\bigl( k_t( a ) \bigr) \wt{\sigma}_{s + t}( 1_\Al \ot z ) \\
& = k_{s + t}( a ) \wt{\sigma}_{s + t}( 1_\Al \ot z ) = %
K_{s + t}( a \ot z ),
\end{align*}
so (ii) holds by linearity and normality. The converse follows
from~\eqref{k = K iota}.
\end{proof}

\begin{rem}
When $k$ is *-homomorphic and ultraweakly continuous we refer to~$K$
as the \emph{corresponding $E$-semigroup on $\Al \otol \Noise$} or,
if $k$ is also unital, and thus a QS flow, the \emph{corresponding
$E_0$-semigroup on $\Al \otol \Noise$}
 (\textit{cf.}~\cite{Arveson}).
\end{rem}

\begin{defn}
If $k$ is a normal completely bounded QS cocycle on $\Al$ then, for
each $c$,~$d\in\noise$, setting
\begin{equation}\label{assoc}
\mathcal{Q}^{c,d}_t( a ) := %
E^{\vp(c_{[ 0, t [})} k_t( a ) E_{\vp(d_{[ 0, t [})}
\quad ( a \in \Al, \, t \in \Rplus )
\end{equation}
defines a normal completely bounded semigroup $\mathcal{Q}^{c,d}$
on~$\Al$;
 $\mathcal{Q}^{0,0} = \bigl( \mathbb{E} \comp k_t \bigr)_{t\ges 0}$
is called the (\emph{vacuum}) \emph{expectation semigroup} of $k$.
The cocycle is called \emph{Markov regular} if each of these
\emph{associated semigroups} is norm continuous.

\end{defn}
\begin{rem}
For a completely bounded cocycle whose cb norm is locally uniformly
bounded, Markov regularity is equivalent to norm continuity of the
expectation semigroup.
\end{rem}

The following result is an amalgamation of the basic existence
theorem for QS differential equations with bounded constant
coefficients with a recent characterisation theorem for QS cocycles.
 In order to state it we introduce some
more notation. For a mapping process~$k$ on~$\Al$ such that $\Dom
k_t( a^* )^* \supset \init \otul \Exps$ for all $a \in \Al$ and $t
\in \Rplus$, setting
\[
k^\dagger_t( a ) := k_t( a^* )^* \bigr|_{\init \otul \Exps} %
\quad ( a \in \Al, \, t \in \Rplus )
\]
defines a mapping process $k^\dagger$ on $\Al$. In this case we say
that $k$ is \emph{adjointable}. The dagger notation is also used as
follows: for Hilbert spaces $\hil$ and $\hil'$, and a map $\psi :
\Al \to B( \hil; \hil' )$, $\psi^\dagger$ is the map
\[
 \Al \to B( \hil'; \hil ), \ a \mapsto \psi( a^* )^*.
\]
Also, for all $\ve \in \Exps$ set
$\iota^{\Al}_{\Fock, \ve} := \iotaAFock( \cdot ) E_\ve$ and
$k_{t, \ve} := k_t( \cdot ) E_\ve$.

\begin{thm}[\cite{LWexistence},\cite{L}]\label{amalgam}
Let $\phi \in CB( \Al; B(\khat) \otol \Al )$. The QS
differential equation
\begin{equation}\label{qsde}
  k_0 = \iotaAFock, \quad %
d k_t = \wt{k}_t \comp \phi \, d \Lambda_t
\end{equation}
has a unique weakly regular, weak solution $k^\phi$. Moreover,
$k^\phi$ is a QS cocycle, it is adjointable with
$( k^\phi )^\dagger = k^{\phi^\dagger}$, and it satisfies
\begin{equation}\label{Holder}
\limsup_{t\to 0^+}
 t^{-1/2}
 \bigl\{
   \norm{k_{t, \ve} - \iota^{\Al}_{\Fock, \ve}}_\cb +
   \norm{k^\dagger_{t, \ve} - \iota^{\Al}_{\Fock, \ve}}_\cb
 \bigr\}
< \infty \quad ( \ve \in \Exps ).
\end{equation}
In particular, $k^\phi$ is continuous and therefore a strong solution
of \eqref{qsde}.

In the converse direction, if $k$ is a completely bounded QS cocycle
on $\Al$ with locally bounded cb norm that satisfies~\eqref{Holder}
then $k = k^\phi$ for a unique map $\phi \in CB( \Al; B( \khat )
\otol \Al )$.
\end{thm}

\begin{rems}
The process $k = k^\phi$ given by Theorem~\ref{amalgam} need not be
bounded. \emph{Weak regularity} means that $\kappa^{f, g}_t$ (given
by~\eqref{kfg}) is bounded with norm locally bounded in~$t$, for all
$f$,~$g \in \Step$; in particular, the terms appearing in
\eqref{weak cocycle} are well defined. To be a \emph{weak solution}
means that
\[
 \langle \zeta, ( k_t - k_0 )( a ) \zeta \rangle = \int_0^t
 \langle \gradhat_s \zeta,
  \wt{k}_s\bigl( \phi( a ) \bigr) \gradhat_s \zeta
 \rangle
\, d s
\ \ %
( \zeta \in \init \otul \Exps, \, a \in \Al, \, t \in \Rplus ),
\]
and to be a \emph{strong solution} means furthermore that $\bigl(
\wt{k}_t( \phi( a ) ) \bigr)_{t\ges 0}$ is QS integrable for all $a
\in \Al$. If $k$ is completely bounded with locally bounded cb~norm,
then $k$ is ultraweakly continuous; in this case, the stochastic
generator~$\phi$ is normal if and only if the cocycle $k$ is normal.

The condition~\eqref{Holder} is stronger than Markov regularity in
general, but they are equivalent when~$k$ is a completely positive
QS cocycle which is \emph{quasicontractive}, that is, there exists
$\beta \in \Real$ such that $\| e^{- \beta t} k_t \| \les 1$ for all
$t \in \Rplus$.
\end{rems}

Necessary and sufficient conditions on the stochastic
generator~$\phi$ for the cocycle~$k^\phi$ to be *-homomorphic are
given next, in terms of a \emph{Lindbladian}
\[
\mathcal{L}_{l, \pi, h} : \Al \to \Al, \ %
a \mapsto l^* \pi( a ) l - \mbox{$\frac12$} ( l^* l a + a l^* l )
+ i ( a h - h a )
\]
and an inner $\pi$-derivation
\[
\delta_{l, \pi} : \Al \to \ket{\Hil} \otol \Al, \ %
a \mapsto \pi( a ) l - l a,
\]
where $l \in \ket{\Hil} \otol \Al$, $h = h^* \in \Al$ and
$\pi : \Al \to B( \Hil \ot \init )$ is a representation of~$\Al$.

\begin{thm}\label{theorem 0}
Let $j = k^\theta$, where $\theta\in CB( \Al; B(\khat) \otol \Al )$.
The following are equivalent\tu{:}
\begin{rlist}
\item
$j$ is normal and *-homomorphic\tu{;}
\item
$\theta = \bigl[\begin{smallmatrix}
 \mathcal{L} & \delta^\dagger \\[0.5ex]
 \delta & \pi - \iota
\end{smallmatrix}\bigr]$,
 where $\iota := \iotaAnoise$,
 $\pi$ is a normal *-homomorphism, $\delta = \delta_{l, \pi}$
and $\mathcal{L} = \mathcal{L}_{l, \pi, h}$ for some $h = h^*\in\Al$
and $l\in \ket{\noise} \otol \Al$.
\end{rlist}
In this case $j$ is ultraweakly continuous; moreover, $j$ is unital
if and only if $\pi$ is unital.
\end{thm}
\begin{proof}
Suppose that (i) holds. The quantum It\^o product
formula~\eqref{QIPF} implies that $\theta$ has the above
block-matrix form, where $\pi$ is a *-homomorphism, $\delta$ is a
$\pi$-derivation and
$\mathcal{L}( a^* b ) - \mathcal{L}( a )^* b - a^* \mathcal{L}( b ) = %
\delta( a )^* \delta( b )$ for all $a \in \Al$. Since $\Ran \delta
\subset \ket{\noise} \otol \Al$, (ii) follows by applying twice the
Christensen--Evans characterisation
 (\cite{ChE}, Theorem~2.1) of a $\rho$-derivation
 $\gamma: \Al \to B( \init; \Kil )$ satisfying
 $\gamma( a )^*\gamma( a ) \in \Al$ for all $a \in \Al$, where
 $\rho : \Al \to B(\Kil )$ is a representation;
 see~\cite{jmlLectures}, Theorem~6.9.

For the reverse implication see Theorem~3.3.6 of~\cite{GS} and
Corollary~4.2 of~\cite{LWhomomorphic}.
\end{proof}

\begin{rem}
 In Theorem 6.8 of~\cite{jmlLectures}
 the hypothesis
 $\gamma( \Al )^* \gamma( \Al ) \subset \Al$ is missing.
\end{rem}

We next describe the situation for QS cocycles which are not
necessarily Markov regular. To any ultraweakly continuous normal
completely positive QS cocycle $k$ on $\Al$ we may associate the maps
 \[
\phi_t : \Al \to B( \khat ) \otol \Al, \ a \mapsto %
\begin{bmatrix}
 t^{-1/2} E^{\vp( 0 )} \\[1ex] V_t^*
\end{bmatrix}
\bigl( k_t( a ) - \iotaAFock( a ) \bigr)
\begin{bmatrix}
 t^{-1 / 2} E_{\vp( 0 )} & V_t
\end{bmatrix}
\]
where $t > 0$ and $V_t$ denotes the isometry
 $\noise \ot \init \to \init \ot \Fock$
 determined by
\begin{equation}\label{V isom}
 c \ot u
 \mapsto u \ot ( 0, t^{-1 / 2} c_{[ 0, t [}, 0, 0, \cdots )
 \qquad
 (u\in\init, c \in \noise),
\end{equation}
and also the operator $\phi$ from $\Al$ to $B(\khat) \otol \Al$,
given by
\begin{equation}\label{phi = uw}
 \phi(x) = \uwlim_{t \to 0^+} \phi_t( x ),
\end{equation}
with domain equal to the set of elements $x \in \Al$ for which the
ultraweak limit exists. Thus, letting $\calDcd$ denote the domain of
the ultraweak generator $\tau_{c,d}$ of the $( c, d )$-associated
semigroup $\mathcal{Q}^{c,d}$ of $k$, $\Dom \phi$ is a *-invariant
subspace of $\Al$ contained in $\calDzerozero$. Set
\begin{equation}\label{D00two}
\calDzerozerotwo :=
\bigl\{ x \in \calDzerozero :
 x^* x, \, x x^* \in \calDzerozero \bigr\}.
\end{equation}

\begin{thm}[\cite{L}]\label{XYZ}
Let $k$ be an ultraweakly continuous normal completely positive QS
cocycle on $\Al$ which is quasicontractive.
\begin{alist}
\item The cocycle $k$ strongly satisfies the following QS
differential equation on $\Dom \phi$,
 where $\phi$ is as above\tu{:}
\begin{equation}\label{QSDE}
k_0 = \iotaAFock, \quad d k_t = \widetilde{k}_t \comp \phi \, d
\Lambda_t.
\end{equation}
\item
 The following relations hold\tu{:}
\begin{align*}
\calDzerozerotwo \subset \Dom \phi & = \bigl\{
 x \in \Al : \limsup_{t \to 0^+} \norm{\phi_t( x )} < \infty
\bigr\} \\
 &
\subset \big\{
 x \in \Al: \liminf_{t\to 0^+} \norm{\phi_t( x )} < \infty
\big\} \subset \bigcap_{c,d\in\noise} \calDcd,
\end{align*}
with the latter two inclusions being equalities when
 $\dim \noise < \infty$.
\item For all $x \in\Dom \phi$,
\begin{equation}
\label{1.24}
 \tau_{c,d}( x ) = E^{\chat} \phi( x ) E_{\dhat} - \chi(
c, d ) x,
\end{equation}
where
$\chi( c, d ) :=
\frac12 \bigl( \norm{c}^2 + \norm{d}^2 \bigr) - \ip{c}{d}$.
\item If $\Dom \phi$ is an ultraweak core for $\tau_{0,0}$
then~\eqref{QSDE} has no other ultraweakly continuous normal
 weak solution which is quasicontractive.
\end{alist}
\end{thm}

 We refer to $\phi$ as the \emph{ultraweak stochastic derivative}
 of~$k$.
\begin{rems}
With the aid of the quantum martingale representation theorem
 (\cite{PS2}), Accardi and Mohari effectively obtained part (a) of
Theorem~\ref{XYZ}, on the domain~$\calDzerozerotwo$,
 assuming that $\init$ and $\noise$ are separable (\cite{AcM}).
There are cases of interest for which the latter domain is
\emph{not} ultraweakly dense (\cite{F2},\cite{ArvDomain}).
\end{rems}

 Weakening the complete positivity assumption, we have the following
 modified version of Theorem~\ref{XYZ}, in which
 the ultraweak stochastic derivative is still defined by~\eqref{phi = uw}
 and the domains $\calDcd$ are defined as before too.

 \begin{thm}[\cite{L}]\label{1.11}
 Let $k$ be an ultraweakly continuous normal completely quasicontractive
  QS cocycle on $\Al$ for which both
  $k$ and $k^\dagger$ are pointwise strong operator continuous, i.e.
  \[
 s \mapsto k_s(a) \xi \text{ and } s \mapsto k_s(a^*)^* \xi
 \text{ are continuous}
 \qquad (a \in \Al, \xi \in \init \ot \Fock).
  \]
  Then
  \begin{alist}
  \item
  $k$ strongly satisfies~\eqref{QSDE} on $\Dom \phi$\tu{;}
  \item
 $\Dom \phi \subset \bigcap_{c,d\in\noise} \calDcd$,
 with equality if $\dim \noise < \infty$,
 and~\eqref{1.24} holds.
  \end{alist}
 \end{thm}

\section{Multiplier cocycles for a QS flow}
\label{section:multipliers}

 \emph{Fix now, and for the rest of the paper}, a QS flow $j$ on $\Al$,
which we refer to as the \emph{free flow} (i.e. unperturbed flow),
and let $J$ denote the corresponding $E_0$-semigroup, according to
Lemma~\ref{k to K}. Thus $J$ is the ultraweakly continuous normal
unital *-homomorphic semigroup on $\Al \otol \Noise$ given by $(
\wh{\jmath}_t \comp \sigma_t )_{t \ges 0}$.

\begin{lemma}\label{uw to SOT}
 For all~$A \in B(\khat)\otol \Al$,
 the map $t \mapsto \wt{\jmath}_t(A)$ is strongly continuous
 $\Rplus \to B(\khat) \otol \Al \otol \Noise$.
 For all~$t\in\Rplus$,
 the map $T \mapsto J_t(T)$ on $\Al \otol \Noise$
 is strongly continuous on bounded sets.
 \end{lemma}
 \begin{proof}
 Straightforward.
 \end{proof}

\begin{defn}\label{defn2.1}
A \emph{multiplier cocycle for $j$}, or \emph{adapted right
$J$-cocycle}, is a bounded process $Y$ in~$\Al$ such that
\[
Y_0 = I_{\init \ot \Fock} \quad \text{and} \quad
Y_{s + t} = J_s( Y_t ) Y_s \quad ( s, t \in \Rplus).
\]
 A multiplier cocycle $Y$ is said to be
 \emph{quasicontractive}
 if there exists $\omega \in \Real$ such that
 $e^{-\omega t} \norm{Y_t} \les 1$ for all $t \in \Rplus$, whereas
 $Y$ is said to have
 \emph{bounded ultraweak stochastic derivative} if there exists
 $F \in B(\wh{\noise}) \otol \Al$ such that
  \begin{equation}\label{Ft to F}
\begin{bmatrix}
 t^{-1/2} E^{\vp( 0 )} \\[1ex] V_t^*
\end{bmatrix}
\bigl( Y_t - I_{\init\ot\Fock} \bigr)
\begin{bmatrix}
 t^{-1 / 2} E_{\vp( 0 )} & V_t
\end{bmatrix}
 \overset{\uw}\longrightarrow F \ \text{ as } t \to 0,
\end{equation}
 where, for $t\in\Rplus$,
 $V_t$ is the isometry defined in~\eqref{V isom}.
\end{defn}

\begin{rems}
A multiplier cocycle for the trivial flow $j \equiv \iotaAFock$ is
precisely a \emph{right QS bounded-operator cocycle}:
\begin{equation}
\label{QS operator cocycle}
X_0 = I_{\init\ot\Fock} \quad \text{and} \quad
X_{s + t} = \wt{\sigma}_s( X_t ) X_s \quad ( s, t \in \Rplus ),
\end{equation}
which is also a process in $\Al$ (\cite{LWcb1}).

If a multiplier cocycle for $j$ is strongly continuous on $\init \ot
\Fock$ then, by the Banach--Steinhaus Theorem, it is necessarily
locally uniformly bounded.

 Stochastic differentiability is relevant to the characterisation of
 perturbation processes for the free flow in
 Section~\ref{section:character}.

A bounded process $Y$ on $\init$ is an adapted right $J$-cocycle if
and only if the adjoint process $Z = ( Y^*_t )_{t \ges 0}$ is an
adapted left $J$-cocycle, that is, $Z$ satisfies $Z_0 = I_{\init \ot
\Fock}$ and $Z_{s + t} = Z_s J_s( Z_t )$ for all $s$,~$t \in
\Rplus$.
\end{rems}

In this paper we work exclusively with right cocycles.

 \begin{rem}
 The free flow is called \emph{inner} if
 it is unitarily implemented by a right QS unitary cocycle
 $U$ in $\Al$, so that
 $j_t = U_t^*( \, \cdot \, \ot I ) U_t$ and
$U_t \in \Al \otol \Noise_{[0,t[} \otol I_{[t,\infty[}$ for all
 $t\in\Rplus$.
 In this case, as is easily verified, the prescription
 \begin{equation} \label{eg bij}
 X \mapsto (U_t^* X_t)_{t\ges 0}
 \end{equation}
 defines a bijection from the class of right QS bounded-operator
 cocycles to the class of multiplier cocycles for $j$.
 \end{rem}

 Thus it is in the case of non-inner free flows
 that we go beyond QS operator cocycles.

 \begin{example}[\cite{LS},\cite{DaL}]
 A large source of such flows is provided by those which induce
 L\'evy processes on $C^*$-bialgebras and compact quantum groups.
 In the latter case, such flows are implemented by QS unitary cocycles; however,
 these cocycles are typically not processes \emph{in}  the algebra.
 A further source is given by the equivariant stochastic flows on
 those spectral triples for which there is currently a good notion of
 quantum isometry group (see~\cite{Gos}).
  \end{example}

\begin{lemma}\label{lemma 2.2}
Let $Y$ be a multiplier cocycle for $j$ and suppose that~$Y$ is
locally uniformly bounded. Then $Y$ is \emph{exponentially
bounded}\tu{:} there exist $M \ges 1$ and $\beta \in \Real$ such
that $\norm{Y_t} \les M e^{\beta t}$ for all $t \in \Rplus$.
\end{lemma}
\begin{proof}
In view of the inequality $\norm{Y_{s + t}} \les \norm{Y_s} \,
\norm{Y_t}$, this follows from a standard semigroup argument; see,
for example, Proposition~1.18 of~\cite{EBD}.
\end{proof}

\begin{rem}
Similarly, a QS cocycle $k$ on $\Al$ is exponentially completely
bounded (in the obvious sense) if its cb~norm is locally uniformly
bounded, since $\cbnorm{k_{s + t}} \les \cbnorm{k_s} \cbnorm{k_t}$.
\end{rem}

\begin{propn}\label{k is cocycle}
Let $Y$ and $Z$ be multiplier cocycles for $j$. Setting
\[
k_t( a ) := Y^*_t j_t( a ) Z_t \quad ( a \in \Al, \, t \in \Rplus)
\]
defines a normal completely bounded QS cocycle $k$ on $\Al$, provided
only that it is pointwise weakly measurable. If $Y$ and $Z$ are
strongly continuous on $\init \ot \Fock$ then~$k$ is ultraweakly
continuous and exponentially completely bounded.
\end{propn}
\begin{proof}
The last part follows from the fact that $Y$ and $Z$ are locally
uniformly bounded when they are strongly continuous, together with
Lemma~\ref{lemma 2.2} and the ultraweak continuity of $j$. By
Lemma~\ref{k to K}, it suffices therefore to show that the normal
completely bounded maps
$\bigl( K_t := \wh{k}_t \comp \sigma_t \bigr)_{t \ges 0}$ form a
semigroup on $\Al \otol \Noise$. The identity
\[
K_s( T ) = Y^*_s J_s( T ) Z_s
\quad ( s \in \Rplus, \, T \in \Al \otol \Noise ),
\]
where $J_s = \wh{\jmath}_s \comp \sigma_s$, is easily verified for a
simple tensor $T$ and so holds in general by linearity and
normality. Therefore, as~$J$ is a *-homomorphic semigroup,
\begin{align*}
K_{s + t}( T )
& = Y^*_s J_s( Y^*_t ) J_s\bigl( J_t( T ) \bigr) J_s( Z_t ) Z_s \\
& = Y^*_s J_s \bigl( Y^*_t J_t( T ) Z_t \bigr) Z_s
 = K_s\bigl( K_t( T ) \bigr)
\quad ( s, t \in \Rplus, \, T \in \Al \otol \Noise ),
\end{align*}
as required.
\end{proof}

\begin{rems}
Thus a pair of multiplier cocycles perturbs the free flow to give a
normal completely bounded QS cocycle on $\Al$
 (modulo the measurability caveat). When
$Z = Y$ the QS cocycle $k = \bigl( Y_t^* j_t( \cdot ) Y_t
\bigr)_{t\ges 0}$ is completely positive; if, also, $Y$ is
coisometric then $k$ is *-homomorphic, and then $k$ is unital, and
thus a QS flow, if and only if $Y$ is furthermore isometric and so
unitary.

When $k$ is ultraweakly continuous, the weak*-generator of its
expectation semigroup is a perturbation of the generator of the
expectation semigroup of the free flow $j$. In
Section~\ref{section:FKformulae} we describe this perturbation in
terms of the stochastic derivative of~$j$, when the multiplier
cocycles $Y$ and $Z$ are governed by QS differential equations.
\end{rems}

 The following is now clear.

 \begin{propn}[\emph{Cf.} the bijection~\eqref{eg bij}]
 \label{injection}
 The prescription
 $Y \mapsto \bigl( j_t( \cdot ) Y_t \bigr)_{t\ges 0}$
 defines an injection from the class of strong operator continuous
 multipler cocycles for $j$ into the class of normal completely
 bounded QS cocycles on $\Al$,
 with left inverse given by
 $k \mapsto \bigl( k_t( 1 ) \bigr)_{t\ges 0}$.
 \end{propn}

\begin{defn}
A \emph{bounded pure-noise QS cocycle in} $\Al$ is a
 (right, equivalently left) QS bounded operator
cocycle in $\Al$ which is $1_{\Al} \ot \Noise$-valued.
\end{defn}

We end this section by noting that bounded pure-noise QS operator
cocycles act on multiplier cocycles.

\begin{propn}\label{pure noise action}
Let $Y$ be a multiplier cocycle for $j$ and let $Z$ be a bounded
pure-noise QS operator cocycle in $\Al$.
 Then $Y Z$ is also a
multiplier cocycle for $j$, provided only that it is weakly
measurable.
\end{propn}
\begin{proof}
Let $s$, $t \in \Rplus$ and set $X = Y Z$. Then
\[
\wt{\sigma}_s( Z_t ) \in 1_\Al \ot I_{[ 0, s [} \ot \Noise_{[ s,
\infty [} \subset \bigl( \Al \otol \Noise_{[ 0, s [} \ot I_{[ s,
\infty [} \bigr)',
\]
so $J_s( Z_t ) = \wt{\sigma}_s( Z_t )$ and
\[
X_{s + t} = J_s( Y_t ) Y_s \wt{\sigma}_s( Z_t ) Z_s
= J_s( Y_t ) J_s( Z_t ) Y_s Z_s = J_s( X_t ) X_s.
\]
Since $X_0 = I$, the result follows.
\end{proof}

A useful example of a bounded pure-noise QS operator cocycle is the
\emph{vacuum-projection cocycle}~$Z$. This is given by
\begin{equation}\label{vac proj cocycle}
Z_t = 1_\Al \ot \ket{\Omega_t}\bra{\Omega_t} \ot I_{[ t, \infty [}
\quad ( t \in \Rplus ),
\end{equation}
where $\Omega_t := \vp(0) \in \Fock_{[ 0, t [}$. Direct verification
confirms that $Z$ strongly satisfies the QS differential equation
\begin{equation}\label{QSDE for vacuum}
  Z_0 = I_{\init \ot \Fock}, \quad
d Z_t = - \Delta \wt{Z}_t d\Lambda_t,
\end{equation}
where $\Delta$ is in ampliated form. This is used in
Proposition~\ref{vacuum shredding}.

\section{Time-dependent QS differential equations}\label{section:qsdes}

In this short section we detail the basic existence theorem required
for the construction of perturbation processes in
Section~\ref{section:perturb}, namely the coordinate-free and
dimension-independent counterpart to Proposition~3.1 of~\cite{GLW}. We
also highlight the H\"older-continuity properties of solutions of such
equations.

\begin{thm}\label{time dep existence}
Let $G$ be a bounded measurable process in $B( \khat ) \otol \Al$
with locally uniform bounds. The QS differential equation
\begin{equation}\label{time dep QSDE W}
  Y_0 = I_{\init \ot \Fock}, \quad
d Y_t = G_t \wt{Y}_t \, d\Lambda_t
\end{equation}
has a unique strong solution. This solution is such that
$Y_s E_\ve \in \Al \otol \ket{\Fock}$ for all $s\in \Rplus$ and
$\ve \in \Exps$, and
\begin{equation}\label{Holder 1}
\sup_{0 \les r < t \les T}
 ( t - r )^{-1 / 2} \norm{( Y_t - Y_r ) E_\ve}
< \infty \quad ( T > 0, \, \ve \in \Exps ).
\end{equation}
\end{thm}
\begin{proof}
This follows from Theorems 10.3 and 10.4 of~\cite{Lqsi}.
\end{proof}

\begin{rem}
The proof of existence consists of a verification that the natural
Picard iteration scheme is well defined and converges, using the
Fundamental Estimate~\eqref{FE}. Uniqueness and the H\"older
estimate also follow from the Fundamental Estimate.
\end{rem}

\begin{propn}\label{propn loc u bdd = st cts}
Let $G$ be as in Theorem~\ref{time dep existence} and suppose that
the unique solution~$Y$ of \eqref{time dep QSDE W} is bounded. Then
the following are equivalent\tu{:}
\begin{rlist}
\item $Y$ has locally uniform bounds\tu{;}
\item $Y$ is strongly continuous on $\init \ot \Fock$.
\end{rlist}
\end{propn}
\begin{proof}
That (ii) implies (i) is an immediate consequence of the
Banach--Steinhaus Theorem; the converse follows from the continuity
of $Y$, as expressed in~\eqref{Holder 1}, and the density of $\init
\otul \Exps$.
\end{proof}

Often the adjoint process satisfies the local H\"older-continuity
condition too.

\begin{propn}\label{time dep Holder proposition}
Let $G$ be as in Theorem~\ref{time dep existence} and suppose that the
unique solution~$Y$ of \eqref{time dep QSDE W} is bounded with locally
uniform bounds. Then
\begin{equation}\label{time dep Holder 2}
\sup_{0 \les r < t \les T}
 ( t - r )^{-1 / 2} \norm{( Y^*_t - Y^*_r ) E_\ve} < \infty
\quad ( \ve \in \Exps, \, T > 0 ),
\end{equation}
provided only that the process $\wt{Y}^* G^*$ is measurable.
\end{propn}
\begin{proof}
Under these hypotheses, $\wt{Y}^* G^*$ is a bounded measurable
process on~$\khat \ot \init$ with locally uniform bounds. Hence
$\wt{Y}^* G^*$ is QS integrable and such that
\[
\int_r^t \wt{Y}^*_s G^*_s \, d \Lambda_s \, E_{\ve} = ( Y^*_t -
Y^*_r ) E_{\ve} \quad ( \ve \in \Exps, \, 0 \les r \les t ).
\]
The Fundamental H\"older Estimate~\eqref{FHE} now gives
\eqref{time dep Holder 2}.
\end{proof}

\begin{rem}
Assuming that $Y$ is bounded with locally uniform bounds, the
process $\wt{Y}^* G^*$ is automatically measurable under either of
the following conditions:
\begin{alist}
\item
$\init$ and $\noise$ are separable;
\item
$Y$ is coisometric and $G^*$ is strongly continuous.
\end{alist}
In the former case this is because of the equivalence of strong and
weak measurability for functions with separable range; in the latter
it follows from the equivalence of weak and strong continuity for
isometry-valued functions.
\end{rem}

\section{Perturbation processes}
\label{section:perturb}

Our aim now is to construct multiplier cocycles for the free
flow~$j$ using QS calculus. If $F \in B( \khat ) \otol \Al$ then
Lemma~\ref{uw to SOT} implies that $G = \bigl( \wt{\jmath}_t( F )
\bigr)_{t\ges 0}$ is a continuous process, to which we may apply
both Theorem~\ref{time dep existence} and Proposition~\ref{time dep
Holder proposition}. This gives the following result.

\begin{thm}\label{existence*}
For all $F \in B( \khat )\otol \Al$ the QS differential equation
\begin{equation}\label{QSDE W*}
  Y_0 = I_{\init\ot\Fock}, \quad
d Y_t = \wt{\jmath}_t( F ) \wt{Y}_t \, d \Lambda_t
\end{equation}
has a unique strong solution. This solution is such that
$Y_s E_\ve \in \Al \otol \ket{\Fock}$ for all $s\in\Rplus$ and
\begin{equation*}\label{Holder 1*}
 \sup_{0 \les r < t \les T}
 ( t - r )^{-1 / 2} \norm{(Y_t - Y_r) E_\ve}
< \infty
\quad ( \ve \in \Exps, \, T > 0 ).
\end{equation*}
If \emph{either} $Y$ is bounded with locally uniform bounds and
$\init$ and $\noise$ are separable, \emph{or} $Y$ is coisometric, it
also satisfies the condition
\begin{equation}\label{Holder 2*}
\sup_{0 \les r < t \les T}
 ( t - r )^{-1 / 2}
 \norm{\bigl( ( Y^*_t - Y^*_r \bigr) E_\ve}
< \infty
\quad ( \ve \in \Exps, \, T > 0 ).
\end{equation}
\end{thm}
 The unique solution will be denoted $\Y{j,F}$;
 see Theorem~\ref{YjF is a right J-cocycle} for more on such
 processes, which we refer to as \emph{perturbation processes}.
 The H\"older
conditions play a r\^ole in the characterisation of multiplier
cocycles for Markov-regular free flows
in~Section~\ref{section:character}.

\begin{propn}\label{Injectivity}
The map $F \mapsto \Y{j,F}$ is injective.
\end{propn}
\begin{proof}
Let $F$, $G \in B( \khat ) \otol \Al$ be such that
$\Y{j,F} = \Y{j,G}$. Then, setting $H = F - G$ and $Y = \Y{j,F}$,
\begin{align*}
0 & = t^{-1}
\ip{u \ve(c_{[ 0, 1 [})}{( \Y{j,F}_t - \Y{j,G}_t ) v \ve(d_{[ 0, 1 [})} \\
 & = t^{-1} \int_0^t \ip{u \ve(c_{[ 0, 1 [})}%
{j_s( E^{\chat} H E_{\dhat})Y_s v \ve(d_{[0,1[})} \, d s \\
& \to e^{\ip{c}{d}} \ip{\chat u}{H \dhat v} \quad \text{as } t \to
0^+ \quad ( u, v \in \init, \, c, d \in \noise ),
\end{align*}
by the continuity of $Y$ and the strong continuity of
$s \mapsto j_s(E^{\chat} H E_{\dhat})$. It follows that $H = 0$, as
required.
\end{proof}

We next investigate the relationship between properties of $F$ and
those of the process $\Y{j,F}$ that it generates; for this the
algebraic structure of the generator needs clarification.

For all $F \in B( \khat ) \otol \Al$ we set $q( F ) := F^* + F + F^*
\Delta F$, where $\Delta$ has been ampliated without changing
notation. If $F$ has the block-matrix form
$\bigl[\begin{smallmatrix}
 k & m \\[0.5ex]
 l & w - 1
\end{smallmatrix}\bigr]$
then
\[
q( F ) = \tinymaths{\begin{bmatrix}
 k^* + k + l^* l & l^* w + m \\[1ex]
 m^* + w^* l & w^* w - 1
\end{bmatrix}}
\quad \mbox{and} \quad
q( F^* ) = \tinymaths{\begin{bmatrix}
 k^* + k + m m^* & m w^* + l^* \\[1ex]
 l + w m^* & w w^* - 1
\end{bmatrix}}.
\]

\begin{propn}\label{q and r}
Let $F \in B( \khat ) \otol \Al$ with the block-matrix form
$\bigl[\begin{smallmatrix}
 k & m \\[0.5ex]
 l & w - 1
\end{smallmatrix}\bigr]$, and let $\beta \in \Real$.
Then the following equivalences hold.
\begin{alist}
\item
\begin{rlist}
\item $q( F ) \les \beta \Delta^\perp$\tu{;}
\item $w$ is a contraction,
$b_1 := \beta 1 - ( k^* + k + l^* l ) \ges 0$ and there is a
contraction $v_1 \in \bra{\noise} \otol \Al$ such that
\begin{equation*}\label{form for m}
m = -l^* w + b_1^{1 / 2} v_1
 ( 1 - w^* w )^{1 / 2}\tu{;}
\end{equation*}
\item $q( F^* )\les \beta \Delta^\perp$\tu{;}
\item $w$ is a contraction,
$b_2 := \beta 1 - ( k + k^* + m m^* ) \ges 0$ and there is a
contraction $v_2 \in \ket{\noise} \otol \Al$ such that
\begin{equation*}\label{form for l}
l = - w m^* + ( 1 - w w^* )^{1 / 2} v_2 b_2.
\end{equation*}
\end{rlist}
\item
\begin{rlist}
\item $q( F ) = 0$\tu{;}
\item $w$ is an isometry, $k^* + k + l^* l = 0$ and $m = -l^* w$.
\end{rlist}
\item
\begin{rlist}
\item
$q( F^* ) = 0$\tu{;}
\item
$w$ is a coisometry, $k + k^* + m m^* = 0$ and $l=-w m^*$.
\end{rlist}
\end{alist}
\end{propn}

\begin{proof}
(a) The equivalence of (i) and (iii) is contained in Theorem A.1
of~\cite{Lqsi}. The other equivalences in (a) follow from the
classical result that an operator block matrix
$\bigl[\begin{smallmatrix}
 a & b\\[0.5ex]
 c & d
\end{smallmatrix}\bigr] \in B(\hil_1 \op \hil_2) \otol \Al$
is nonnegative if and only if $c=b^*$, $a \ges 0$, $d \ges 0$ and
 $b = a^{1 / 2} v d^{1 / 2}$ for a contraction
 $v \in B(\hil_2; \hil_1) \otol \Al$;
 see p.~547 of~\cite{FoF} or Lemma 2.1 of~\cite{GLSW}. Parts (b) and (c) are
immediate consequences of the definition of $q$.
\end{proof}

 Set
$\CkA(\noise, \Al) := %
\bigcup_{\beta \in \Real} \CbetakA(\noise, \Al)$,
where
\[
\CbetakA(\noise, \Al) := %
\bigl\{ F \in B( \khat ) \otol \Al :
 q( F ) \les \beta \Delta^\perp \bigr\};
\]
 For
$F \in B( \khat ) \otol \Al$ and $\beta \in \Real$, the identity
\[
q( F - \mbox{$\frac12$} \beta \Delta^\perp ) = %
q( F ) - \beta \Delta^\perp
\]
implies that $F \in \CbetakA(\noise, \Al)$ if and only if $F -
\frac{\beta}{2} \Delta^\perp \in \CzerokA(\noise, \Al)$.

\begin{cor}\label{cor CkA}
Let $F \in B( \khat ) \otol \Al$ with the block-matrix form
$\bigl[\begin{smallmatrix}
 k & m \\[0.5ex]
 l & w - 1
\end{smallmatrix}\bigr]$.

\begin{alist}
\item The following are equivalent\tu{:}
\begin{rlist}
\item $F \in \CkA(\noise, \Al)$\tu{;}
\item
$w$ is a contraction and $m = -l^* w + r_1 ( 1 - w^*w )^{1 / 2}$ for
some operator $r_1 \in \bra{\noise} \otol \Al$\tu{;}
\item
 $w$ is a contraction and
$l = -w m^* + ( 1 - w w^* )^{1 / 2} r_2$ for some operator
$r_2 \in \ket{\noise} \otol \Al$.
\end{rlist}
\item If $m = -l^* w$ or $l = -w m^*$ then the following are
equivalent\tu{:}
\begin{rlist}
\item $F \in \CkA(\noise, \Al)$\tu{;}
\item $w$ is a contraction.
\end{rlist}
\item If $w = 0$ then $F \in \CkA(\noise, \Al)$. More specifically, if $w = 0$
then, for all $\beta \in \Real$ the following are equivalent\tu{:}
\begin{rlist}
\item $F \in \CbetakA(\noise, \Al)$\tu{;}
\item $\beta 1 \ges k^* + k + l^* l + m m^*$.
\end{rlist}
\end{alist}
\end{cor}
\begin{proof}
(a) follows from part (a) of Proposition~\ref{q and r} and (b) follows
from (a).

(c) If $w = 0$ and $\beta \in \Real$ then
\[
\beta \Delta^\perp - q( F ) =
\begin{bmatrix}
 \beta 1 - ( k + k^* + l^* l ) & -m \\[1ex]
  -m^* & 1
\end{bmatrix},
\]
which is nonnegative if and only if
$\beta 1 - ( k^* + k + l^* l ) \ges m m^*$.
\end{proof}

The relevance of the above observations is seen in the next theorem.

\begin{thm}\label{existence 2}
Let $Y = \Y{j,F}$, where $F\in B( \khat ) \otol \Al$, and let
$\beta \in \Real$.
\begin{alist}
\item
The following are equivalent\tu{:}
\begin{rlist}
\item $\bigl( e^{-\beta t / 2} Y_t \bigr)_{t \ges 0}$ is
contractive\tu{;}
\item $F \in \CbetakA(\noise, \Al)$.
\end{rlist}
 Thus $Y$ is quasicontractive
 \tu{(}in the sense of Definition~\ref{defn2.1}\tu{)}
  if and only if
 $F \in \CkA(\noise, \Al)$.
\item
The following are equivalent\tu{:}
\begin{rlist}
\item $Y$ is isometric\tu{;}
\item $q( F ) = 0$.
\end{rlist}
\item
If $Y$ is coisometric then $q( F^* ) = 0$.
\end{alist}
\end{thm}
\begin{proof}
Recalling our convention that $\Step$ consists of right-continuous
step functions, we use the unbounded operator $R$ from $\init \ot
\Fock$ to $\khat \ot \init$ determined by the conditions
\[
\Dom R = \init \otul \Exps \quad \text{and} \quad %
R u \vp( f ) = \fhat( 0 ) u.
\]
(a)\! \&\! (b) In view of the remarks above and the fact that
$\bigl( e^{-\beta t / 2} \Y{j,F}_t \bigr)_{t \ges 0}$ equals
$\Y{j,G}$, where $G = F - \frac12 \beta \Delta^\perp$, to prove (a)
it suffices to assume that $\beta = 0$. Let $\zeta \in \init \otul
\Exps$ and define the function
\[
\varphi: \Rplus \to \Comp, \ s \mapsto
\ip{\wt{Y}_s \gradhat_s \zeta}%
{\wt{\jmath}_s\bigl( q( F ) \bigr) \wt{Y}_s \gradhat_s\zeta}.
\]
Then $\varphi$ is right continuous and, by the Second Fundamental
Formula~\eqref{SFF},
\[
\norm{Y_t \zeta}^2 - \norm{\zeta}^2 = \int_0^t \varphi( s ) \, d s
\quad ( t \in \Rplus ).
\]
Thus $q( F ) \les 0$ implies that $Y$ is contractive and $q( F ) = 0$
implies that $Y$ is isometric. For the converse, note that
\[
t^{-1} \bigl( \norm{Y_t \zeta}^2 - \norm{\zeta}^2 \bigr) \to
\varphi(0^+) = \ip{R \zeta}{q( F ) R \zeta}
\]
as $t \to 0^+$ and $R$ has dense range $\khat \otul \init$. This
proves (a) and (b).

(c) Assume that $Y$ is coisometric. Then $Y^*$ is weakly continuous
and isometric, and so is a continuous process, and
$Y^*_t = I_{\init \ot \Fock} + %
\int_0^t \wt{Y}^*_s \wt{\jmath}_s( F^* ) \, d \Lambda_s$
for all $t \in \Rplus$. Therefore, arguing as above,
\[
\norm{Y_t^* \zeta}^2 - \norm{\zeta}^2 = \int_0^t \psi( s ) \, d s
\quad ( t \in \Rplus, \, \zeta \in \init \otul \Exps),
\]
for the function
\[
\psi: \Rplus \to \Comp, \ s \mapsto
\ip{\gradhat_s \zeta}%
{\wt{\jmath}_s\bigl( q( F^* ) \bigr) \gradhat_s \zeta}.
\]
Hence
\[
0 = t^{-1} \bigl( \norm{Y^*_t \zeta}^2 - \norm{\zeta}^2 \bigr) \to
\psi(0^+) = \ip{R \zeta}{q( F^* ) R \zeta}
\]
as $t \to 0^+$. Using the density of the range of $R$ once more, it
follows that~$q( F^* ) = 0$.
\end{proof}

\begin{rem}
Theorem~\ref{existence 2} contains a coordinate-free and
dimension-independent extension of Propositions~3.3 and~3.4
of~\cite{GLW}; it will be applied in the proof of
Theorem~\ref{Feynman-Kac theorem}.
\end{rem}

When the free flow $j$ is Markov regular, the converse of part (c) of
Theorem~\ref{existence 2} holds. We prove this next; the argument is a
little more involved than that for the converse of part (b), requiring
more than a simple application of the quantum It\^o product formula.

\begin{propn}\label{coisometry under MR}
Let $Y = \Y{j,F}$, where $F\in B( \khat ) \otol \Al$, and suppose
that~$j$ is Markov regular. The following are equivalent\tu{:}
\begin{rlist}
\item $Y$ is coisometric\tu{;}
\item $q( F^* )=0$.
\end{rlist}
\end{propn}
\begin{proof}
Because of Theorem~\ref{existence 2}, it remains only to prove that
(ii) implies (i). Suppose therefore that (ii) holds.
 Following Proposition~3.6 of~\cite{GLW}, define processes
\[
Z := ( Y_t Y^*_t - I )_{t\ges 0}, \quad
k := ( j_t( \cdot ) Y_t )_{t\ges 0} \quad \mbox{and} \quad
l := \bigl( Z_t j_t( \cdot ) \bigr)_{t \ges 0}.
\]
Part (a) of Theorem~\ref{existence 2} implies that $Y$ is contractive;
hence $Z$ is too and $k$ and $l$ are completely contractive
processes. We now show that $l$ is the zero process and so $Y$ is
coisometric by the unitality of $j$. By Theorem~\ref{amalgam}, there
exists a map $\theta \in CB( \Al; B( \khat ) \otol \Al )$ such that
$j = k^\theta$, and the quantum It\^o product formula \eqref{QIPF}
implies that $k = k^\psi$ for the completely bounded map
\begin{equation}\label{the map psi}
\psi : a \mapsto \theta( a ) + \iota( a ) F + \theta( a ) \Delta F,
\quad \text{where} \quad \iota = \iotaAkhat.
\end{equation}
Thus, since $k^\dagger = k^{\psi^\dagger}$, we have
 $Y^*_t = k^{\psi^\dagger}_t( 1 )$ for all $t \in \Rplus$;
 in particular, the
operator process $Y^*$ is continuous and such that
\[
Y^*_t = I +
\int_0^t
 \wt{k^{\dagger}_s}\bigl( \psi^\dagger( 1 ) \bigr) \,
d \Lambda_s
= I + \int_0^t \wt{Y}^*_s \wt{\jmath}_s( F^* ) \, d \Lambda_s
\quad ( t \in \Rplus ).
\]
It follows that $Z$ is continuous and, by the quantum It\^o product formula,
\[
Y_t Y^*_t = I + \int_0^t \Bigl\{
 \wt{\jmath}_s( F ) \wt{Y}_s \wt{Y}^*_s +
 \wt{Y}_s \wt{Y}^*_s \wt{\jmath}_s( F^* ) +
 \wt{\jmath}_s ( F \Delta) \wt{Y}_s \wt{Y}^*_s
 \wt{\jmath}_s ( \Delta F^* )
\Bigr\} \, d\Lambda_s
\]
for all $t \in \Rplus$. The assumption $q( F^* ) = 0$ now implies that
\[
Z_t = \int_0^t \Big\{
 \wt{\jmath}_s( F ) \wt{Z}_s + \wt{Z}_s \wt{\jmath}_s( F^* ) +
 \wt{\jmath}_s( F \Delta ) \wt{Z}_s \wt{\jmath}_s( \Delta F^* )
\Big\} \, d \Lambda_s \quad ( t \in \Rplus ).
\]
A further application of the product formula yields the identity
\[
l_t( a ) = \int_0^t \Bigl\{
 \wt{\jmath}_s( F ) \wt{l}_s\bigl( \mu_1( a ) \bigr) +
 \wt{l}_s\bigl( \mu_2( a ) \bigr) +
 \wt{\jmath}_s( F \Delta ) \wt{l}_s\bigl( \mu_3( a ) \bigr)
\Bigr\} \, d\Lambda_s
\]
for all $t \in \Rplus$ and $a \in \Al$, where $\mu_1 = \mu$,
$\mu_2 = F^* \mu( \cdot ) + \theta$ and
$\mu_3 = \Delta F^* \mu( \cdot )$ for the map
\[
\mu : a \mapsto I_\khat \ot a + \Delta \theta( a ).
\]
Hence, by the complete boundedness and normality of all the maps
involved, for any Hilbert space $\hil$ the process
$l^\hil := \bigl( \id_{B( \hil )} \otol l_t \bigr)_{t \ges 0}$
satisfies
\[
l^\hil_t( A ) = %
\int_0^t \Bigl\{
 \wt{j_s^\hil}( F^\hil ) \wt{l_s^\hil} \bigl( \mu_1^\hil( A ) \bigr) +
 \wt{l_s^\hil}\bigl( \mu_2^\hil( A ) \bigr) +
 \wt{j_s^\hil}\bigl( ( F \Delta )^\hil \bigr)
 \wt{l_s^\hil}\bigl( \mu_3^\hil( A ) \bigr)
\Bigr\} \, d \Lambda_s
\]
for all $t \in \Rplus$ and $A \in B( \hil ) \otol \Al$, where
$j^\hil := \bigl( \id_{B( \hil )} \otol j_t \bigr)_{t \ges 0}$ and,
in terms of the flip
$\Sigma : %
B( \hil \ot \khat ) \otol \Al \to B( \khat \ot \hil ) \otol \Al$,
the operator
$G^\hil := %
\Sigma( I_\hil \ot G )$ for $G \in B( \khat ) \otol \Al$ and the map
$\mu_i^\hil := \Sigma \comp ( \id_{B( \hil )} \otol \mu_i )$ for
$i = 1$, $2$, $3$. Therefore, by the Fundamental Estimate
after~\eqref{FE},
\[
\norm{l^\hil_t( A ) \eta \vp(f)}^2 \les 3 C\bigl( [ 0, t [, f
\bigr)^2 \int_0^t
 \max_i
  \bigl\|
   \wt{l^\hil_s}\bigl( \nu_i^\hil( A ) \bigr) \fhat( s ) \eta \vp( f )
  \bigr\|^2
d s
\]
for all $t \in \Rplus$, $A \in B( \hil ) \otol \Al$, $\eta \in \hil
\ot \init$ and $f \in \Step$, where $\nu_1 = \norm{F} \mu_1^\hil$,
$\nu_2 := \mu_2^\hil$ and $\nu_3 := \norm{F \Delta} \mu_3^\hil$.
Since $\wt{l^\hil} = l^{\khat \ot \hil}$, this estimate may be
iterated. Using the complete contractivity of $l$, $n$ iterations
yield the inequality
\[
\norm{l_t( \cdot ) E_{\vp( f )}} \les
\bigl\{
 \sqrt{3} \, C\bigl( [ 0, t [, f \bigr)
 \max_i \norm{\nu_i}_\cb \norm{\fhat_{[ 0, t [} }
\bigr\}^n / \sqrt{n!}
\]
for all $t \in \Rplus$ and $f \in \Step$; letting $n \to \infty$ shows
that the left-hand side is identically zero. Thus $l = 0$, as
required.
\end{proof}
\begin{rem}
This extends Proposition 3.6 of~\cite{GLW} and rectifies the proof
given there, which neglected a term in its equation (3.10).
\end{rem}

The next few results are relevant to the role of multiplier cocycles
in providing multipliers for Feynman--Kac perturbations of the
expectation semigroup of the free flow $j$. They also provide a link
to the vacuum-adapted approach to such perturbations (\cite{BLS}).

\begin{propn}
 \label{vacuum shredding}
Let $F\in B(\khat)\otol \Al$ with block-matrix form
$\bigl[\begin{smallmatrix}
 k & m \\[0.5ex]
 l & w - 1
\end{smallmatrix}\bigr]$
and let $Z$ be the vacuum projection
cocycle~\eqref{vac proj cocycle}. Then $\Y{j,F} Z = \Y{j,F'}$, where
$F' = F \Delta^\perp - \Delta = %
\bigl[\begin{smallmatrix}
 k & 0 \\[0.5ex]
 l & -1
\end{smallmatrix}\bigr]$.
\end{propn}
\begin{proof}
Set $X = \Y{j,F} Z$. Then $X$ inherits continuity from $\Y{j,F}$,
thus $\wt{X}$ is QS integrable. Recall that $Z$ strongly satisfies
the QS differential equation~\eqref{QSDE for vacuum}. For all $t \in
\Real$, in view of the unitality of $j$, we have $\Delta \ot I_\Fock
= \wt{\jmath}_t( \Delta )$. Therefore, by the quantum It\^o formula,
the process $X$ satisfies the QS differential equation
\[
d X_t = %
\bigl( \wt{\jmath}_t( F ) \wt{X}_t - ( \Delta \ot I_\Fock ) \wt{X}_t - %
\wt{\jmath}_t( F ) ( \Delta \ot I_\Fock ) \wt{X}_t \bigr) %
\, d \Lambda_t = %
\wt{\jmath}_t (F') \wt{X}_t %
\, d \Lambda_t
\]
 with $X_0 = I$. It follows from uniqueness (in Theorem~\ref{existence*})
that $X = \Y{j,F'}$.
\end{proof}

As an immediate consequence we have the following corollary.

\begin{cor}\label{F to F1}
If $F \in B( \khat ) \otol \Al$ then
\[
\Y{j,F}_t E_{\vp( 0 )} = \Y{j,F \Delta^\perp}_t E_{\vp( 0 )} \quad (
t \in \Rplus ).
\]
In particular, if $F$ has the block-matrix form
$\bigl[\begin{smallmatrix} k & m \\[0.5ex] l & n
\end{smallmatrix}\bigr]$ then, for any contraction
 $w \in B( \noise ) \otol \Al$,
\[
\Y{j,F}_t E_{\vp( 0 )} = \Y{j,G}_t E_{\vp( 0 )}, \quad \text{ where
} \quad G = \begin{bmatrix}
 k & -l^* w \\[1ex]
 l & w - 1
\end{bmatrix}.
\]
\end{cor}

\begin{rems}
Note that boundedness of the process $\Y{j,F}$ is not assumed here.
However $G \in \CkA(\noise, \Al)$, by part (b) of Corollary~\ref{cor
CkA}, so $\Y{j,G}$ is quasicontractive, by Theorem~\ref{existence
2}.

The following two instances of Corollary~\ref{F to F1} are of
particular relevance: for $F \in B( \khat ) \otol \Al$, we have
$\Y{j,F} E_{\vp( 0 )} = %
\Y{j,F'} E_{\vp( 0 )} = \Y{j,F''} E_{\vp( 0 )}$,
where
\begin{equation}\label{F2 G2}
F' = \begin{bmatrix}
 k & 0 \\[1ex]
 l & - 1
\end{bmatrix}
\quad \text{and} \quad
F'' = \begin{bmatrix}
 k & -l^* \\[1ex]
 l & 0
\end{bmatrix}.
\end{equation}
\end{rems}

\section{Bounded perturbation processes are multiplier cocycles}
\label{section:bounded pert}

In this section we show that, when bounded with locally uniform
bounds, the perturbation processes constructed in
Section~\ref{section:perturb} are multiplier cocycles for the free
flow $j$.

The following lemma is an immediate consequence of adaptedness
when~$j$ is implemented by a unitary QS operator cocycle, as it is
in~\cite{LiS} and~\cite{BP}. Recall the associated $E_0$-semigroup $J$
of the free flow, defined at the start of
Section~\ref{section:multipliers}.

\begin{lemma}\label{J and QS commute}
Let $X$ be a process in $B( \khat ) \otol \Al$ which is strongly
continuous on $\khat \ot \init \ot \Fock$ and such that the process
$\bigl( \int_0^t X_s \, d\Lambda_s \bigr)_{t \ges 0}$ is also
bounded. For all $t \ges r \ges 0$ and $R \in B( \init ) \otol
\Noise_{[ 0, r [} \ot I_{[ r, \infty [}$,
\[
J_r \Bigl( \int_0^{t - r} X_u \, d\Lambda_u \Big) R =
\int_r^t \wt{J}_r( X_{s - r} ) \wt{R} \, d\Lambda_s,
\]
where $\wt{R} := I_{\khat} \ot R$ and $\wt{J}_r$ denotes the unital
*-homomorphism $\id_{B( \khat )} \otol J_r$.
\end{lemma}
\begin{proof}
The right-hand side is well defined since
$\bigl( \wt{J}_r( X_{s - r} ) \bigr)_{s \ges r}$ is strongly
continuous and adapted. If $R \in B( \init ) \ot I_\Fock$ the result
follows from the extended First Fundamental Formula,
Lemma~\ref{FFF extension}, together with the identity
\[
\Omega[ f, g ] \comp J_r = \Omega[ f_{r)}, g_{r)} ] \comp j_{r)}
\comp
 \Omega[ S_r^* f, S_r^* g ],
\]
which holds for all $f$, $g \in \Step$ and $r \in \Rplus$; here
$f_{r)}$ and $g_{r)}$ denote the restrictions of the functions $f$
and $g$ to the interval $[ 0, r [$, $(S_r)_{r\ges 0}$ is the
right-shift semigroup on $L^2( \Rplus; \noise )$ and
$\Omega[ f, g ] := \id_\Al \otol \omega_{\vp( f ), \vp( g )}$. The
case of general $R$ now follows from the First Fundamental Formula,
taking $R$ to be an ampliated Weyl operator at first.
\end{proof}

\begin{rem}
If $\init$ and $\noise$ are separable then the continuity assumption
on the process $X$ may be replaced by locally uniform boundedness.
\end{rem}

\begin{thm}\label{YjF is a right J-cocycle}
Let $F \in B( \khat ) \otol \Al$ and suppose that the process
$\Y{j,F}$ is bounded with locally uniform bounds. Then $\Y{j,F}$ is
a multiplier cocycle for~$j$.
\end{thm}
\begin{proof}
Fix $r > 0$. We must show that
\begin{equation}\label{YjF eqn 1}
Y_t = J_r( Y_{t - r} ) Y_r
\end{equation}
for all $t > r$, where $Y = \Y{j,F}$. To this end, define a process
$Z$ in $\Al$ by letting
\[
Z_s := \left\{
\begin{array}{ll}
Y_s & \text{if } s \les r, \\
J_r(Y_{s - r}) Y_r & \text{if } s \ges r.
\end{array} \right.
\]
It follows from Proposition~\ref{propn loc u bdd = st cts} that $Y$
is strongly continuous on $\init \ot \Fock$ so, by Lemma~\ref{uw to
SOT}, the bounded process $Z$ is strongly continuous on $\init \ot
\Fock$ too. Hence $\bigl( \wt{\jmath}_s( F ) \wt{Z}_s \bigr)_{s \ges
0}$ is strongly continuous and thus QS integrable. Therefore, by
Lemma~\ref{J and QS commute} and the cocycle property of $j$,
\begin{align*}
Z_t & = J_r(Y_{t - r}) Y_r \\
& = Y_r +
 J_r\Bigl(
  \int_0^{t - r} \wt{\jmath}_u( F ) \wt{Y}_u \, d \Lambda_u
 \Bigr)
 Y_r \\
& = I + \int_0^r \wt{\jmath}_s( F ) \wt{Y}_s \, d \Lambda_s +
\int_r^t
 \wt{\jmath}_s( F ) \wt{J}_r( \wt{Y}_{s - r} ) \wt{Y}_r \,
d \Lambda_s \\
& = I + \int_0^t \wt{\jmath}_s( F ) \wt{Z}_s \, d \Lambda_s,
\end{align*}
for all~$t \ges r$. As this also holds for $t < r$ it follows, by
the uniqueness part of Theorem~\ref{existence*}, that $Z = Y$.
Therefore~\eqref{YjF eqn 1} holds, as required.
\end{proof}

By Proposition~\ref{k is cocycle} and Theorem~\ref{existence 2}, the
theorem above has the following immediate consequence.

\begin{cor}\label{jFG}
Let $F_1$, $F_2 \in B( \khat ) \otol \Al$ and suppose that $\Y{j,F_1}$
and $\Y{j,F_2}$ are bounded with locally uniform bounds. Then the
normal completely bounded process
\begin{equation}\label{j F G}
 \jFonetwo := %
 \bigl( ( \Y{j,F_1}_t )^* j_t( \cdot ) \Y{j,F_2}_t \bigr)_{t \ges 0}
\end{equation}
is an ultraweakly continuous QS cocycle which is exponentially
completely bounded. Moreover, if $F_1$, $F_2 \in \CkA(\noise, \Al)$
then the QS cocycle $\jFonetwo$ is completely quasicontractive.
\end{cor}

We abbreviate $j^{F, F}$ to $\jF$.

\begin{rem}
The completely positive QS cocycle $\jF$ is contractive if and only
if $F \in \CzerokA(\noise, \Al)$; it is unital if and only if
$\Y{j,F}$ is isometric or, equivalently, if $F$ has the block-matrix
form
\[
\begin{bmatrix}
 i h -\frac{1}{2} l^* l & -l^* w \\[1ex]
 l & w - 1
\end{bmatrix},
\quad \mbox{with $h$ selfadjoint and $w$ isometric};
\]
the QS cocycle $\jF$ is *-homomorphic if the multiplier $\Y{j,F}$ is
coisometric, for which a \emph{necessary} condition is that $F$ have
the block-matrix form
\[
\begin{bmatrix}
 i h -\frac{1}{2} m m^* & m \\[1ex]
 -w m^* & w - 1
\end{bmatrix},
\quad \mbox{with $h$ selfadjoint and $w$ coisometric}.
\]
\end{rem}

The next result provides considerable freedom in the choice of
multiplier cocycle for obtaining a given perturbation of the
expectation semigroup of the free flow $j$.

\begin{thm}\label{theorem X}
For $i=1$, $2$, let $F_i \in B( \khat ) \otol \Al$ and suppose that
$\Y{j,F_i}$ is bounded; let $F'_i$ and $F_i''$ be defined as
in~\eqref{F2 G2}. Then the expectation semigroups of the ultraweakly
continuous QS cocycles $j^{F_1', F_2'}$ and~$j^{F_1'', F_2''}$ are
both equal to that of $\jFonetwo$, namely $\bigl( \mathbb{E} \comp
\jFonetwo_t \bigr)_{t \ges 0}$.
\end{thm}
\begin{proof}
By part (b) of Corollary~\ref{cor CkA}, the operators $F_1'$ and
$F_2' \in \CkA(\noise, \Al)$; thus Theorems~\ref{existence 2}
and~\ref{YjF is a right J-cocycle} imply that $\Y{j,F_1'}$ and
$\Y{j,F_2'}$ are quasicontractive multiplier cocycles for~$j$.
Therefore $j^{F_1', F_2'}$ is ultraweakly continuous, by
Corollary~\ref{jFG}, and so its expectation semigroup is too; by
Corollary~\ref{F to F1}, the expectation semigroups of $\jFonetwo$
and $j^{F_1', F_2'}$ are equal. Exactly the same argument applies
to~$j^{F_1'', F_2''}$; the result follows.
\end{proof}

\begin{rem}
In Section~\ref{section:FKformulae} we shall see how the
weak*-generator of the expectation semigroup of $\jFonetwo$ appears
as a perturbation of that of the free flow $j$ on the domain of its
QS derivative.
\end{rem}

\section{Characterisations of perturbation processes}
\label{section:character}

 In this section we give two results converse to
 Theorem~\ref{YjF is a right J-cocycle}.
 The first is a characterisation of perturbation processes
 whose adjoint process is strongly continuous.
 The second leads to a characterisation theorem for perturbation
 processes under the assumption that the free flow is Markov
 regular, simultaneously generalising results of Bradshaw and of
 Lindsay and Wills.

 Quasicontractive perturbation processes are strongly continuous,
 since they strongly satisfy a QS differential equation,
 and it is not hard to verify that,
 in the sense of~\eqref{Ft to F}, the perturbation process
 $\Y{j,F}$ has ultraweak stochastic derivative at $0$ equal to $F$.
 Conversely,
 recalling the injection described in Proposition~\ref{injection},
 we have the following.

 \begin{thm} \label{6.3*}
 Let $Y$ be a strongly continuous quasicontractive multiplier
 cocycle for $j$ whose adjoint process $Y^*$ is strongly continuous.
 Suppose that $Y$ has a bounded ultraweak stochastic derivative.
 Then
 $Y$ is a perturbation process.
 \end{thm}
 \begin{proof}
 Let $F$ be the stochastic derivative of $Y$ at $0$.
 The associated QS cocycle $k := ( j_t( \cdot ) Y_t )_{t\ges 0}$
 is normal and completely quasicontractive;
 moreover, $k$ and $k^\dagger$ are both pointwise strongly
 continuous.
 Since $k_t(1) = Y_t$ ($t\ges 0$),
 letting $\phi$ be the ultraweak stochastic derivative of $k$,
 as defined in~\eqref{phi = uw},
 \[
 \phi_t(1) =
 \begin{bmatrix}
 t^{-1/2} E^{\vp( 0 )} \\[1ex] V_t^*
\end{bmatrix}
\bigl( Y_t - I_{\init\ot\Fock} \bigr)
\begin{bmatrix}
 t^{-1 / 2} E_{\vp( 0 )} & V_t
\end{bmatrix}
 \overset{\uw}\longrightarrow F \ \text{ as } t \to 0^+.
 \]
 Thus $1 \in \Dom \phi$ and $\phi(1) = F$.
 Therefore, by Theorem~\ref{1.11},
 $k_t(1) = I + \int_0^t \wt{k}_s(\phi(1)) d\Lambda_s$,
 in other words
 $Y_t = I + \int_0^t \wt{\jmath}_s(F)\wt{Y}_s d\Lambda_s$
 ($t\ges 0$).
 The result now follows by uniqueness (in Theorem~\ref{existence*}).
 \end{proof}
 \begin{rem}
 Since the adjoint process of a quasicontractive perturbation
 process is strongly continuous when $\init$ and $\noise$ are
 separable,
 the above result gives a characterisation of such perturbation
 processes in the separable case.
 \end{rem}

 We next show that,
 under a H\"older-regularity assumption, every multiplier cocycle of
 the free flow $j$ satisfies a QS differential equation
 of the form~\eqref{QSDE W*}.

\begin{thm}\label{partial converse}
Let $Y$ be a locally uniformly bounded multiplier cocycle for~$j$
and suppose that the associated QS cocycle $k := ( j_t( \cdot ) Y_t
)_{t\ges 0}$ satisfies the condition
\begin{equation}\label{Holder again}
\limsup_{t\to 0^+} t^{-1 / 2} \bigl\{
 \norm{k_{t, \ve} - \iota^{\Al}_{\Fock, \ve}}_\cb +
 \norm{k^\dagger_{t, \ve} - \iota^{\Al}_{\Fock, \ve}}_\cb \bigr\} < \infty
\quad ( \ve \in \Exps ).
\end{equation}
Then $Y = \Y{j,F}$ for a unique operator $F \in B( \khat ) \otol
\Al$.
\end{thm}
\begin{proof}
As $k$ is completely bounded with locally bounded cb norm,
 $k = k^\phi$ for a completely
bounded map~$\phi : \Al \to \Al \otol B( \khat )$, by
Theorem~\ref{amalgam}. Since $Y_t = k_t( 1_\Al )$, it follows that
$Y$ strongly satisfies the QS differential equation~\eqref{QSDE W*},
with $F = \phi(1_\Al)$, and so $Y = \Y{j,F}$ by uniqueness.
Uniqueness of~$F$ follows from Proposition~\ref{Injectivity}.
\end{proof}

\begin{rems}
When the free flow $j$ is Markov regular, \eqref{Holder again} is
equivalent to
\begin{equation}\label{Y Holder}
\limsup_{t\to 0^+} t^{-1/2} \norm{( Y_t - I ) E_{\ve} } < \infty
\text{ and } \limsup_{t\to 0^+} t^{-1/2} \norm{( Y^*_t - I )
E_{\ve}} < \infty \ \ ( \ve \in \Exps ).
\end{equation}
This is because, in this case, $j$ itself satisfies the
inequality~\eqref{Holder again}.

The first of the conditions in~\eqref{Y Holder} is necessary for~$Y$
to be of the form~$\Y{j,F}$; the second is too if~$Y$ is coisometric
or if~$Y$ is locally uniformly bounded and the spaces~$\init$
and~$\noise$ are separable, by Theorem~\ref{existence*}.
\end{rems}

\begin{thm}\label{Y under MR}
Let $Y$ be a quasicontractive multiplier cocycle for the free
flow~$j$ and suppose that $j$ is Markov regular. The following are
equivalent\tu{:}
\begin{rlist}
\item $Y = \Y{j,F}$ for an operator $F \in B( \khat )\otol \Al$\tu{;}
\item the associated QS cocycle $k := ( j_t( \cdot ) Y_t )_{t \ges 0}$
is Markov regular.
\end{rlist}
In this case $F = \psi( 1 )$, where $\psi$ is the stochastic
generator of~$k$.
\end{thm}
\begin{proof}
If (i) holds then, as shown in the proof of
Proposition~\ref{coisometry under MR}, $k = k^\psi$ for the completely
bounded map~\eqref{the map psi} and so (ii) holds. Conversely,
suppose that (ii) holds. Recall that $e^{\beta t} Y^F_t = Y^G_t$,
where $G = F + \beta \Delta^\perp$. Therefore, replacing $Y$ by
$( e^{-\beta t} Y_t )_{t \ges 0}$ for a suitable $\beta \in \Real$, we
may assume without loss of generality that $Y$ is contractive. Fixing
$u \in \init$ and $f \in \Step$, let $\mathcal{Q}$ be the
$( c, c )$-associated semigroup of $k$, as defined in~\eqref{assoc},
where $c = f( 0 )$. Then, by the contractivity of $Y$, when $t$ is
less than any point of discontinuity of the step function $f$,
\begin{align*}
 \norm{( Y_t - I ) u \vp( f )}^2& +
 \norm{( Y^*_t - I ) u \vp( f )}^2 \\
& \les 4 \re \ip{u \vp( f )}{( I - Y_t ) u \vp( f )} \\
& = 4
 \re \ip{u}{\bigl( 1 - \mathcal{Q}_t( 1 ) \bigr) u}
\les 4\, \norm{ \id_\Al - \mathcal{Q}_t } \, \norm{u}^2.
\end{align*}
As $k$ is Markov regular, $\mathcal{Q}$ is norm continuous and thus
norm differentiable.
 Therefore~\eqref{Y Holder} holds, and the result follows
from Theorem~\ref{partial converse} and the first remark following
it.
\end{proof}

\begin{rem}
This result generalises both Theorem~6.7 of~\cite{LWjfa} (the case
where $j$ is trivial, $Y$ is contractive, $\Al = B( \init )$ and
$\noise$ is separable) and the main result of~\cite{Bra} (the case
where $Y$ is unitary, $\init$ is separable and $\noise$ has one
dimension). Bradshaw uses the quantum martingale representation
theorem (\cite{PS1}); his proof does not extend beyond the case of
unitary~$Y$ and separable~$\init$.
\end{rem}

\section{Perturbations of flows}\label{section:pert flows}

In~\cite{LiS} and~\cite{BP} the free flow $j$ is given by
\[
j_t = \al_{B_t} \quad ( t \in \Rplus ),
\]
where $( \al_s )_{s \in \Real}$ is a normal *-automorphism group on
$\Al$ and $( B_t )_{t \ges 0}$ is a standard one-dimensional
Brownian motion. In this case $j$ is governed by the classical
stochastic differential equation
\[
d j_t = j_t \comp \delta \, d B_t
+ \mbox{$\frac12$} j_t \comp \delta^2 d t,
\]
where $\delta$ is the derivation generating $\al$. In these papers
it is further assumed that the automorphism group, and so also the
free flow, is unitarily implemented. Here we generalise to consider
a genuinely QS flow $j$, which is neither assumed to be unitarily
implemented nor assumed to be driven by a classical Brownian motion.

Let $\theta$ be the ultraweak stochastic derivative of the free
flow~$j$ given by Theorem~\ref{XYZ}, and let $\tau$ be the ultraweak
generator of the expectation semigroup of $j$. Then $\Dom \theta$ is
a subspace of $\Al$ contained in $\Dom \tau$,
 $1_{\Al} \in \Dom \theta$ with $\theta( 1_{\Al} ) = 0$ and the map $\theta$ is real
(that is, $\Dom \theta$ is *-invariant and $\theta( x )^* = \theta(
x^* )$ for all $x \in \Dom \theta$). By the quantum It\^o product
formula,
\begin{equation}\label{theta structure}
\theta( x^* y ) = \theta( x )^* \iota( y ) + %
\iota( x )^* \theta( y ) + \theta( x )^* \Delta \theta( y ),
\end{equation}
where $\iota = \iota^{\Al}_{\khat}$, for all
$x$, $y \in \Dom \theta$ such that $x^* y \in \Dom \theta$. In terms
of the block-matrix form
$\left[\begin{smallmatrix}
 \mathcal{L} & \delta^\dagger \\[0.5ex]
 \delta & \pi - \iota
\end{smallmatrix}\right]$
 of $\theta$,
 in which $\iota = \iota_\noise^\Al$,
 the structure relations~\eqref{theta structure} are equivalent to
 the conditions
\begin{align*}
 \pi( x^* y ) & = \pi( x )^* \pi( y ), \\
 \delta( x^* y ) & = \delta( x )^* y + \pi( x )^* \delta( y ) \\
 \text{and} \quad \mathcal{L}( x^* y ) & =
 \mathcal{L}( x )^* y + x^*\mathcal{L}( y ) + \delta( x )^* \delta( y )
\end{align*}
for all $x$, $y \in \Dom \theta$ such that $x^* y \in\Dom \theta$.

\begin{thm}\label{bounded perturbation}
Let $( j, \theta )$ be as above, let $F_1$, $F_2 \in B( \khat )
\otol \Al$ and suppose that the perturbation process~$\Y{j,F_1}$ is
bounded. Then the mapping process~$\jFonetwo$, defined by~\eqref{j F
G}, weakly satisfies the QS differential equation
\begin{equation}\label{k psi}
  k_0 = \iota, \quad
d k_t = \wt{k}_t \comp \phi \, d\Lambda_t,
\end{equation}
where $\Dom \phi = \Dom \theta$, $\iota = \iota^\Al_\Fock$ and, for
all $x \in \Dom \phi$,
\[
\phi( x ) =
\theta( x ) +
 F_1^* \bigl( \Delta \theta( x ) + \iota( x ) \bigr) +
 F_1^* \Delta \bigl( \theta( x ) + \iota( x ) \bigr) \Delta F_2 +
 \bigl( \theta( x ) \Delta +\iota( x ) \bigr) F_2.
\]
If the adjoint process $\bigl( ( \Y{j,F_1}_t )^* \bigr)_{t \ges 0}$ is
strongly continuous on $\init \ot \Fock$ then~$\jFonetwo$
satisfies~\eqref{k psi} strongly.
\end{thm}
\begin{proof}
The first part is a straightforward consequence of
 the quantum It\^o product formula \eqref{QIPF} and
 the Second Fundamental Formula \eqref{SFF},
 each in their polarised form,
 applied respectively to
 $j_t(x) \Y{j,F_2}_t = j_t(x^*)^* \Y{j,F_2}_t$
 and
 $(\Y{j,F_1}_t)^* j_t(x) \Y{j,F_2}_t$.
 If $\bigl( ( \Y{j,F_1}_t )^* \bigr)_{t \ges 0}$ is
strongly continuous on~$\init \ot \Fock$ then, setting $k =
\jFonetwo$, the process $\wt{k}\bigl( \phi( x ) \bigr)$ is QS
integrable for each $x \in \Dom \phi$ and so the second part follows
too.
\end{proof}

\begin{rems}
With regard to the hypotheses, recall that if $F \in \CkA(\noise,
\Al)$ then $\Y{j,F}$ is quasicontractive, by Theorem~\ref{existence
2}; if, in addition, $\init$ and $\noise$ are separable then $(
\Y{j,F} )^*$ is necessarily strongly continuous on $\init \ot
\Fock$, by Theorem~\ref{existence*}.

If $F_2 = F_1 = F$, $q(F) = 0$ and $\Y{j,F}$ is coisometric (for
which the necessary condition $q(F^*) = 0$ is also sufficient when
$j$ is Markov regular) then
 $\Y{j,F}$ is unitary, so $\jF$ is a QS flow on $\Al$.

If $F_1$ and $F_2$ have block-matrix forms
\[
 \begin{bmatrix}
 k_1 & m_1 \\[1ex]
 l_1 & w_1 - 1
\end{bmatrix} \quad \mbox{and} \quad
 \begin{bmatrix}
 k_2 & m_2 \\[1ex]
 l_2 & w_2 - 1
\end{bmatrix},
\]
respectively, then $\phi( x )$ has block matrix form
\[
\begin{bmatrix}
 \tinymaths{\mathcal{L}( x ) +
 l_1^* \delta( x ) + l_1^* \pi( x ) l_2 + \delta^\dagger( x ) l_2 +
 k_1^* x + x k_2}
 & \tinymaths{\bigl( \delta^\dagger( x )
  + l_1^* \pi( x ) \bigr) w_2 + x m_2} \\[1ex]
 \tinymaths{w^*_1 \bigl( \delta( x ) + \pi( x ) l_2 \bigr) + m_1^* x}
 & \tinymaths{w_1^* \pi( x ) w_2 - x \ot 1}
\end{bmatrix}.
\]
\end{rems}

\section{Feynman--Kac formulae}\label{section:FKformulae}

The cocycle-perturbation theorem of the previous section yields a
general form of quantum Feynman--Kac formula. Let $\mathcal{P}^0$ be
a Markov semigroup on~$\Al$
 with weak*-generator $\tau$.
 Suppose that $\mathcal{P}^0$ is
 realised as the expectation semigroup of a QS flow $j$, in the sense
that $\mathcal{P}^0_t = \mathbb{E} \comp j_t$ for all $t \in
\Rplus$, with $\mathbb{E}$ the vacuum expectation defined
in~\eqref{vacuum}. Let $\theta = \left[\begin{smallmatrix}
 \mathcal{L} & \delta^\dagger \\[0.5ex]
 \delta & \pi - \iota
\end{smallmatrix}\right]$
be the QS derivative of $j$, as in Theorem~\ref{XYZ}.
 Thus $\mathcal{L} \subset \tau$ and $\iota = \iotaAnoise$.

\begin{thm}\label{Feynman-Kac theorem}
Let $( j, \theta )$ be as above, and let $l_1$, $l_2 \in
\ket{\noise} \otol \Al$ and $k_1$, $k_2 \in \Al$. Then there is an
ultraweakly continuous normal completely quasicontractive semigroup
$\mathcal{P}$ on $\Al$ whose weak*-generator is an extension of the
operator with domain $\Dom \theta$ given by the prescription
\begin{equation}
\label{gen P}
 x \mapsto \mathcal{L}( x ) +
 l_1^* \delta( x ) +
 l_1^*\pi( x ) l_2 +
 \delta^\dagger( x ) l_2 +
 k_1^* x + x k_2.
\end{equation}
 Moreover, $\mathcal{P}$ is
\begin{alist}
\item contractive if $k_1^* + k_1 + l_1^*l_1 \les 0$ and
$k_2^* + k_2 + l_2^*l_2 \les 0$\tu{;}
\item unital if $k_1^* + l_1^* l_2 + k_2 = 0$\tu{;}
\item completely positive if $l_1 = l_2$ and $k_1 = k_2$.
\end{alist}
\end{thm}
\begin{proof}
Let
\[
F_1 =
\begin{bmatrix}
 k_1 & -l^*_1 \\[1ex]
 l_1 & 0
\end{bmatrix}
\quad \text{and} \quad
F_2 =
\begin{bmatrix}
 k_2 & -l^*_2 \\[1ex]
 l_2 & 0
\end{bmatrix}.
\]
By Corollary~\ref{cor CkA}, $F_1$, $F_2 \in \CkA(\noise, \Al)$ and
so, by Corollary~\ref{jFG}, $\jFonetwo$ is an ultraweakly continuous
normal completely quasicontractive QS cocycle; let~$\mathcal{P}$ be
its expectation semigroup and let $\phi$ be the map defined in
Theorem~\ref{bounded perturbation}. Then, for $x \in \Dom \theta$,
 $E^{\wh{0}} \phi(x) E_{\wh{0}}$ equals~\eqref{gen P},
 so the first part follows from Theorem~\ref{1.11}.
 Moreover,
 $\jFonetwo$ is unital if
$\phi( 1_\Al ) = 0$, completely positive if $F_1 = F_2$ and
contractive if $q(F_1) \les 0$ and $q(F_2) \les 0$, by
Theorem~\ref{existence 2}. Since
\[
\phi( 1_\Al ) = \Delta^\perp \ot ( k_1^* + l_1^* l_2 + k_2 )
\quad \text{and} \quad
q( F_i ) = \Delta^\perp \ot ( k_i^* + k_i + l_i^*l_i )
\]
for $i = 1,2$, the result follows from Theorem~\ref{bounded
perturbation}.
\end{proof}

\begin{rems}
The term $k^*_1 x + x k_2$ contributes a bounded perturbation which
can alternatively be realised via the Trotter product formula; see,
for example,~\cite{EBD}.

The \emph{domain algebra} of $\mathcal{P}^0$ is the largest unital
*-subalgebra of $\Al$ contained in $\Dom \mathcal{L}$, whose existence
follows from an argument using the Kuratowski--Zorn Lemma. Note that
the domain algebra is contained in the selfadjoint unital subspace
$\calDzerozerotwo$ of $\Al$ defined in~\eqref{D00two}. There are
interesting examples in which the domain algebra fails to be
ultraweakly dense in $\Al$, both commutative (\cite{F2}) and
noncommutative (\cite{ArvDomain}), and interesting examples where it
is ultraweakly dense (\cite{ArvDomain}). In the latter case $\Dom
\theta$ is ultraweakly dense, by Theorem~\ref{XYZ}. When $\Al = B(
\init )$ and $\init$ is separable, $\calDzerozerotwo$ is an algebra
 (\cite{ArvDomain}) and therefore equals the domain algebra
of~$\mathcal{P}^0$.
\end{rems}

Thus any QS flow $j$ dilating $\mathcal{P}^0$, in the sense that
$\mathcal{P}^0 = \big( \mathbb{E} \comp j_t \big)_{t\ges 0}$, and
any $F_1$, $F_2 \in B( \khat ) \otol \Al$ such that $\Y{j,F_1}$ is
bounded, give rise to a semigroup~$\mathcal{P}$ whose generator is a
noncommutative vector field-type perturbation of the generator of
$\mathcal{P}^0$, through the quantum Feynman--Kac formula
\[
\mathcal{P}_t = \mathbb{E}\Bigl[
 ( \Y{j,F_1}_t )^* j_t( \cdot ) \Y{j,F_2}_t
\Bigr] \quad ( t \in \Rplus ).
\]
In terms of the perturbation generators $F_1$ and $F_2$, the
semigroup~$\mathcal{P}$ depends only on $F_1 \Delta^\perp$ and
$F_2 \Delta^\perp$.

Specialising to Markov semigroups we have the following result, which
considerably extends the class of Feynman--Kac formulae obtained
in~\cite{BP}, where the free flow is obtained from an automorphism
group of $\Al$ randomised by a classical Brownian motion.

 \begin{cor} \label{FK cpu}
Let $( j, \theta )$ be as above, and let $l \in \ket{\noise} \otol
\Al$ and $h = h^* \in \Al$. There is a completely positive Markov
semigroup $\mathcal{P}$ on $\Al$ whose ultraweak generator is an
extension of the operator with domain $\Dom \theta$ given by the
prescription
\[
x \mapsto \mathcal{L}( x ) + l^* \delta( x ) + l^* \pi( x ) l + %
\delta^\dagger( x ) l + k^* x + x k,
\]
where $k := i h - \frac{1}{2} l^* l$.
\end{cor}

 \emph{ACKNOWLEDGEMENTS.}
 We are grateful to an anonymous referee for prompting us to provide
 (in Section~\ref{section:character})
 a characterisation of perturbation processes beyond the case of
 Markov-regular free flows,
 and to Kalyan Sinha for his interest at an early stage of the work.
 Support from the UK-India Education and Research Initiative (UKIERI)
 is also gratefully acknowledged.


\end{document}